\documentclass[A4paper,english,11pt]{article}
\usepackage[headings]{fullpage}
\usepackage{amsmath,amssymb,amsthm}
\newcommand{\myappendix}{Appendix~} 

\usepackage[verbose]{wrapfig}
\usepackage{subfig,xcolor,graphicx}

\ifx \eqref \undefined
\newcommand{\eqref}[1]{(\ref{#1})}
\fi

\newtheorem{thrm}{Theorem}
\newtheorem{lem}[thrm]{Lemma}
\newtheorem{cor}[thrm]{Corollary}

\newtheorem{conj}[thrm]{Conjecture}

\newcommand{\bvec}[1] {{\mathbf {#1}}}

\newcommand{\BPC}{{P_{\text{cr}}^{\text{\tiny B}}}}

\newcommand{\BR}{{R_{\text{\tiny B}}}}

\newcommand{\BS}{{S_{\text{\tiny B}}}}

\newcommand{\TCrit}{{T_{\rm c}}}

\newcommand{\Schrodinger}{{Schr\"odinger }}

\ifx \DeclareMathOperator \undefined

\else

\fi

\newcommand{\mycaption}[1]{\parbox{0.95\textwidth}{\caption{{#1}}}}

\newcommand{\Real}{{\mathbb R}}
\newcommand{\abs}[1]{{\left\vert{#1}\right\vert}}
\newcommand{\norm}[1]{{\left\Vert{#1}\right\Vert}}

\newcommand{\BSlin}{{S_{\text{\tiny B}}^{lin}}}

\graphicspath{{}}

\usepackage{hyperref}

\begin{document}

\renewcommand{\thesubfigure}{{\Alph{subfigure}}}

\title{Singular solutions of the $L^2$-supercritical biharmonic %
Nonlinear Schr\"odinger equation}

\makeatletter
\author{{G. Baruch$^*$}, {G. Fibich},\\
    School of Mathematical Sciences, Tel Aviv University, Tel Aviv 69978,
    Israel \\
    $^*$Corresponding author, guy.baruch@math.tau.ac.il
}
\makeatother
\maketitle

\begin{abstract}
    We use asymptotic analysis and numerical simulations to study peak-type
    singular solutions of the supercritical biharmonic NLS.
    These solutions have a quartic-root blowup rate, and collapse with a quasi
    self-similar universal profile, which is a zero-Hamiltonian solution of a
    fourth-order nonlinear eigenvalue problem.
\end{abstract}

\section{\label{sec:intro}Introduction}

The focusing nonlinear \Schrodinger equation (NLS)
\begin{equation}    \label{eq:NLS}
    i\psi_t(t,\bvec{x}) + \Delta\psi + \left|\psi\right|^{2\sigma}\psi = 0,
    \qquad \psi(0,\bvec{x}) = \psi_0(\bvec{x})\in H^1(\Real^d),
\end{equation}
where~$
    \bvec x=\left(x_1,\dots,x_d\right)\in\Real^d
$ and~$
    \Delta = \sum_{j=1}^d\partial_j^2
$ is the Laplacian, has been the subject of intense study, due to its role in
various areas of physics, such as nonlinear optics and Bose-Einstein Condensates
(BEC).
It is well-known that the NLS~\eqref{eq:NLS} possesses solutions that become
singular in a finite time~\cite{Sulem-99}.

In recent years, there has been a growing interest in extending NLS theory to the focusing biharmonic nonlinear \Schrodinger equation
(BNLS)
\begin{equation}    \label{eq:BNLS}
    i\psi_t(t,\bvec{x}) - \Delta^2\psi + \left|\psi\right|^{2\sigma}\psi = 0,
    \qquad \psi(0,\bvec{x}) = \psi_0(\bvec{x})\in H^2(\Real^d),
\end{equation}
where~$\Delta^2$ is the biharmonic operator.
The BNLS~\eqref{eq:BNLS} is called~``{\em $L^2$-critical}'', or simply
``critical'' if~$\sigma d=4$.
In this case, equation~\eqref{eq:BNLS}  can be rewritten as
\begin{equation}    \label{eq:CBNLS}
    i\psi_t(t,\bvec{x}) - \Delta^2\psi
        +\left|\psi\right|^{8/d}\psi  =  0,
    \qquad \psi(0,\bvec{x}) = \psi_0(\bvec{x})\in H^2(\Real^d).
\end{equation}
Correspondingly, the BNLS with~$0<\sigma d<4$ is called subcritical, and the BNLS 
with~$\sigma d>4$ is called supercritical.
This is analogous to the NLS, where the critical case is~$\sigma d=2$.

BNLS solutions preserve the {\em power} ($L^2$~norm)
$$
   P(t) \equiv P(0), \qquad P = ||\psi||_2^2,
$$
and the Hamiltonian 
$$
   H(t) \equiv H(0), \qquad H = ||\Delta \psi||_2^2 - \frac{1}{\sigma+1} ||\Delta \psi||_{2 \sigma+2}^{2 \sigma +2}.
$$
In~\cite{Ben-Artzi-00}, Ben-Artzi, Koch and Saut proved that when~$\sigma$ is in
the \emph{$H^2$-subcritical} regime
\begin{equation}    \label{eq:admissible-range}
    \begin{cases}
        0<\sigma & \quad d\leq4,\\
        0<\sigma<\frac4{d-4} & \quad d>4,
    \end{cases}
\end{equation} 
 the BNLS~\eqref{eq:BNLS} is locally well-posed in~$H^2$.
Global existence and scattering of BNLS solutions in the $H^2$-critical 
case~$\sigma=4/(d-4)$ were studied by Miao, Xu and Zhao~\cite{Miao20093715} and by
Pausader~\cite{Pausader_DCDS-2009a}.
The latter work also showed well-posedness for small data.
The $H^2$-critical defocusing BNLS was studied by Miao, Xu and Zhao~\cite{miao-2008}
and by Pausader~\cite{Pausader_DPDEs-2007,Pausader_JFA2009}.

The above studies focused on non-singular solutions.
In this work, we study singular solutions of the BNLS in~$H^2$, i.e., solutions that
exist in~$H^2(\Real^d)$ over some finite time interval~$t\in[0,\TCrit)$, 
but for which~$
    \displaystyle \lim_{t\to\TCrit} \norm{\psi}_{H^2} = \infty.
$
The first study of singular BNLS solutions was done by Fibich, Ilan and
Papanicolau~\cite{Fibich_Ilan_George_BNLS:2002}, who proved the following results:
\begin{thrm} \label{thrm:GE_subcritical}
    Let $\psi_0 \in H^2$. Then, the solution of the subcritical BNLS~\eqref{eq:BNLS} exists
    globally in~$H^2$.
\end{thrm} 
\begin{thrm}    \label{thrm:GE_critical}
    Let $\psi_0 \in H^2$, and let $ \norm{\psi_0}_2^2 < \BPC$, where~$\BPC = \norm{\BR}_2^2$, and~$\BR$ is
    the ground state of
\begin{equation}    \label{eq:stationary_state_critical}
    -\Delta^2 \BR(\bvec{x}) - \BR + |\BR|^{8/d}\BR = 0.
\end{equation}
    Then, the solution of the critical BNLS~\eqref{eq:CBNLS} exists
    globally in~$H^2$.
\end{thrm} 

\noindent
The simulations in~\cite{Fibich_Ilan_George_BNLS:2002} suggested that there
exist singular solutions for~$\sigma d=4$ and~$\sigma d>4$, and that these
singularities are of the blowup type, namely, the solution becomes infinitely
localized.
However, in contradistinction with NLS theory, there is currently no rigorous
proof that solutions of the BNLS can become singular in either the critical or
the supercritical case.

Most subsequent research of singular BNLS solutions focused on the critical case. 
 Chae, Hong and Lee~\cite{ChaeHongLee:2009}, showed that radial
singular solutions of~\eqref{eq:CBNLS} have a power-concentration property.
In~\cite{Baruch_Fibich_Mandelbaum:2009a}, we showed that radial singular solutions
are quasi self-similar.
We also proved, without assuming radial symmetry, that the blowup rate is bound by a
quartic-root, the power-concentration property and the existence of the
ground-state of~\eqref{eq:stationary_state_critical}.
The two latter properties were also proved by Zhu, Zhang and
Yang~\cite{ZhuZhangYang:2010}.
In~\cite{Baruch_Fibich_Mandelbaum:2009a}, we also provided informal
analysis and numerical evidence that peak-type singular solutions of
the critical BNLS collapse with a quasi self-similar $\BR$~profile 
at a blowup rate which is slightly faster than the quartic-root bound.

In this work, we use asymptotic analysis and numerics to find and characterize
peak-type singular solutions of the supercritical BNLS. We find that their properties mirror those of the supercritical NLS.
Ring-type singular solutions of the supercritical BNLS were studied in~\cite{Baruch_Fibich_Gavish:2009,Baruch_Fibich_Mandelbaum:2009b}.

\subsection{\label{ssec:intro_summary}Summary of results}

We analyze singular solutions of the focusing
$L^2$-supercritical and $H^2$-subcritical BNLS, i.e., when 
\begin{equation}    \label{eq:admissible-range-L+H}
    \begin{cases}
        4/d<\sigma & d\leq4,\\
        4/d<\sigma<\frac4{d-4} & d>4.
    \end{cases}
\end{equation} 
We assume radial symmetry, i.e., that~$\psi = \psi(t,r)$,
where $ r=|{\bvec x}|$.
In this case, equation~\eqref{eq:BNLS} reduces to 
\begin{equation}    \label{eq:radial_BNLS}
    i\psi_t(t,r) - \Delta^2_r\psi + \left|\psi\right|^{2\sigma}\psi = 0,
    \qquad \psi(0,r) = \psi_0(r),
\end{equation}
where
$$
    \Delta_r^2 = 
        \partial_r^4 
        +\frac{2(d-1)}{r}\partial_r^3
        +\frac{(d-1)(d-3)}{r^2}\partial_r^2
        -\frac{(d-1)(d-3)}{r^3}\partial_r
$$
is the radial biharmonic operator.

The paper is organized as follows.
In Section~\ref{sec:explicit}, we show that the supercritical BNLS admits 
the explicit self-similar singular solutions 
\begin{equation}    \label{eq:intro_SS_supercrit}
    \psi_\BS(t,r) =
        \frac{1}{L^{2/\sigma}(t)}
        \BS\left( \frac r{L(t)} \right)
        e^{i \nu \int \frac{1}{L^4(t')}dt'},
\end{equation} 
where the blowup rate of~$L(t)$ is exactly a quartic-root
$$
    L(t)=\kappa(\TCrit-t)^{1/4},    \qquad \kappa>0,
$$ 
and the self-similar profile~$\BS(\rho)$ is a solution of 
\begin{equation}
   \label{eq:S_BBB}
\begin{gathered}
    -\nu \BS(\rho) + i\frac{\kappa^4}4 \left(
        \frac{2}{\sigma}\BS + \rho \BS^\prime 
    \right)
    - \Delta_\rho^2 \BS + \abs{\BS}^{2\sigma}\BS = 0,
    \\
    \BS^\prime(0)=\BS^{\prime\prime\prime}(0)=0,
    \qquad
    \BS(\infty)=0.
\end{gathered}
\end{equation}
WKB analysis of the large-$\rho$ behavior of~$\BS$ shows that it belongs to~$L^{2+2\sigma}$,
but not to~$L^2$.  Since 
    $\lim_{t\to\TCrit} \norm{\psi_\BS}_{2+2\sigma} = \infty
$, $\psi_\BS$ is a singular solution in~$L^{2+2\sigma}$, but not in~$H^2$.
To the best of our knowledge, this is the first time that explicit singular solutions
of the BNLS are presented.

In Section~\ref{ssec:explicit_OH} we show that the zero-Hamiltonian solutions
of~\eqref{eq:S_BBB} satisfy the boundary condition
\[
    \lim_{\rho\to\infty}\left( 
        \rho S^\prime + \left( 
            \frac 2\sigma + i \frac{4\nu}{\kappa^4}
        \right)S
    \right)\rho^{\gamma}
    =0,
\qquad        
\frac23\left( d-2-\frac2\sigma \right) <\gamma<4+\frac2\sigma. 
\]
In analogy with the supercritical NLS, we conjecture that for any~$d$, $\sigma$
and~$\nu$, there is unique {\em admissible} solution~$\BS^{\rm admis.}(\rho)$, which has a
zero Hamiltonian and is monotonically decreasing. 
This solution is attained for a unique $\kappa = \kappa^{\rm admis.}(\sigma,d,\nu)>0$.
While a rigorous existence proof for the~$\BS$ profile remains open, we provide
numerical support for the existence of the admissible solutions.

In Section~\ref{sec:asymptotic} we consider $H^2$~singular solutions.
Using informal asymptotic analysis and the analogy with the supercritical NLS,
we conjecture that these solutions undergo a quasi self-similar collapse with the
$\psi_\BS$~profile, where~$\BS$ is the unique admissible solution~$\BS^{\rm admis.}$.
The blowup rate of these solutions is given by
$L(t) \sim \kappa^{\rm admis.} (\TCrit -t)^{1/4}$. 
These characteristics are confirmed numerically, in simulations of both the one-dimensional and the two-dimensional BNLS.

The numerical simulations of the BNLS were performed using the IGR/SGR 
method~\cite{Ren-00,SGR-08}, see~\cite{Baruch_Fibich_Mandelbaum:2009b} for further details.
The numerical solution of the nonlinear fourth-order ODE for~$\BS$ is obtained
using a modified Petviashvili (SLSR) method, which is described in the
appendix.
The code is available online at 
{\em http://www.math.tau.ac.il/$\sim$fibich/publications.html}

The results of this study are based on asymptotic analysis and numerical simulations,
but not on rigorous analysis.
These results show that there is a striking analogy between collapse of
peak type solutions in the supercritical NLS and the supercritical BNLS. 
We note that the rigorous theory for singular solutions of the supercritical NLS is
much less developed than that for the critical NLS.
Indeed, a rigorous proof of the blowup rate and blowup profile of the
supercritical NLS was obtained very recently, and only in the slightly-supercritical
regime $0<\sigma d-2 \ll 1$ \cite{Szeftel-09}.
We hope that this study will motivate a similar rigorous treatment of the supercritical
BNLS.

\section{Explicit singular solutions}
\label{sec:explicit}

Let us look for explicit self-similar solutions of the supercritical
BNLS~\eqref{eq:BNLS}.
Since the BNLS is invariant under the dilation symmetry~$
    r\mapsto \frac rL,
    t\mapsto \frac t{L^4},
    \psi\mapsto\frac{1}{L^{2/\sigma}}\psi
$, where~$L$ is a constant, this suggests a self-similar solution of the form
\begin{equation}    \label{eq:psi_SB}
    \psi_\BS(t,r) = 
        \frac1{L^{2/\sigma}(t)} 
        \BS\left( \rho \right)
        e^{i\tau(t)},
    \qquad  \rho = \frac r{L(t)}.
\end{equation}
Substituting~$\psi_\BS$ in the BNLS gives
\begin{equation}    \label{eq:psi_SB_pre_ODE}
    -\tau^\prime(t) L^4(t)\BS(\rho)
    -iL^3(t)L_t\left(
        \frac 2\sigma \BS
        + \rho\BS^\prime
    \right)
    - \Delta_{\rho}^2\BS
    +\abs{\BS}^{2\sigma}\BS = 0.
\end{equation}
Since~$\BS$ is only a function of~$\rho$,
equation~\eqref{eq:psi_SB_pre_ODE} must be independent of~$t$.
Therefore, there exists a real constant~$\kappa$ such that
\[
    L^3L_t \equiv\frac14\left( L^4 \right)_t \equiv - \kappa^4/4.
\]
Hence,~$L(t)$ is a quartic root, i.e.,
\begin{equation}    \label{eq:L_14}
    L(t) = \kappa \sqrt[4]{\TCrit-t},
    \qquad \kappa > 0.
\end{equation}
Likewise, since~$\BS$ is only a function of~$\rho$,
then~$\tau^\prime(t) L^{4}(t)\equiv \nu$.
Hence,
\begin{equation}    \label{eq:tau}
    \tau(t) = \nu \int_{s=0}^{t}\frac{1}{L^4(s)}ds
    = - \frac{\nu}{\kappa^4} \ln\left( 1-\frac t\TCrit \right) .
\end{equation}

    Substituting~\eqref{eq:L_14} and~\eqref{eq:tau}
    in~\eqref{eq:psi_SB_pre_ODE} shows that the equation for~$\BS$ is
\begin{subequations}    
   \label{eqs:explicit_ODE}
    \begin{equation}    \label{eq:explicit_ODE_profile}
        -\nu \BS(\rho) + i\frac{\kappa^4}4 \left(
            \frac{2}{\sigma}\BS + \rho \BS^\prime 
        \right)
        - \Delta_\rho^2 \BS + \abs{\BS}^{2\sigma}\BS = 0.
    \end{equation}
       Since~$\BS$ is  radially-symmetric and decays at infinity, it should satisfy the boundary conditions
        \begin{equation}    \label{eq:explicit_ODE_BCs}
            \BS^\prime(0)=\BS^{\prime\prime\prime}(0)=0,
            \qquad
            \BS(\infty)=0.
        \end{equation}
\end{subequations}

Equation~\eqref{eqs:explicit_ODE} has the two parameters~$\nu$ and~$\kappa$.
Note, however, that 
 \begin{equation}   \label{eq:Sb_tildeSb}
    \BS(\rho; \kappa,\nu):= 
    \nu^{\frac1{2\sigma}} \tilde{\BS}(
            \nu^{1/4} \rho; \tilde{\kappa} =\kappa/\nu^{1/4}
        ),
    \end{equation}
where $\tilde{\BS}(\rho; \tilde{\kappa})$ is the solution of~\eqref{eqs:explicit_ODE} with $\nu=1$, i.e.,  
\begin{subequations}    
   \label{eqs:explicit_ODE_nu=1}
    \begin{equation}    \label{eq:explicit_ODE_profile-nu=1}
        - \tilde{\BS}(\rho) + i\frac{\kappa^4}4 \left(
            \frac{2}{\sigma}\tilde{\BS} + \rho \tilde{\BS}^\prime 
        \right)
        - \Delta_\rho^2 \tilde{\BS} + \abs{\tilde{\BS}}^{2\sigma}\tilde{\BS} = 0.
    \end{equation}
subject to
        \begin{equation}    \label{eq:explicit_ODE_nu=1_BCs}
            \tilde{\BS}^\prime(0)=\tilde{\BS}^{\prime\prime\prime}(0)=0,
            \qquad
            \tilde{\BS}(\infty)=0.
        \end{equation}
\end{subequations}

Equation~\eqref{eqs:explicit_ODE_nu=1} can be viewed as a nonlinear eigenvalue problem
with the eigenvalue~$\kappa$ and eigenfunction~$\BS$.
By analogy with the supercritical NLS~\cite{Koppel-95,Budd-01},
we make the following conjecture:
    \begin{conj}    \label{conj:explicit}
        Let $\sigma$  be in the $L^2$~supercritical and~$H^2$-subcritical regime~\eqref{eq:admissible-range-L+H}.
        Then, there exists a solution~$\left\{ 
        \tilde{\BS}(\rho),\tilde{\kappa} \right\}$ to 
        equation~\eqref{eqs:explicit_ODE_nu=1}, such that~$\tilde{\BS} \not\equiv 0$ and~$\tilde\kappa>0$.
    \end{conj}

Hence, we have the following result:
\begin{lem} \label{lem:explicit}
    Assume that Conjecture~\ref{conj:explicit} holds, and let~$\BS(\rho;\kappa,\nu)$ be a nontrivial solution of~\eqref{eqs:explicit_ODE}.
    Then,
    \begin{equation}    \label{eq:psiBS}
        \psi_\BS(t,r) = 
        \frac1{L^{2/\sigma}(t)} 
        \BS\left( \frac r{L(t)} \right)
        e^{i \nu \int^{t}\frac{1}{L^4(s)}ds},
        \qquad  L(t)=\kappa\sqrt[4]{\TCrit-t},
    \end{equation}
    is an explicit solution of the BNLS equation~\eqref{eq:BNLS}.
\end{lem}

As $\rho\to\infty$, the nonlinear term in~\eqref{eqs:explicit_ODE} becomes
negligible, and~\eqref{eq:explicit_ODE_profile} reduces to
\begin{equation}    \label{eq:SC_WKB_profile}
    -\nu \BSlin(\rho) 
    -\Delta_\rho \BSlin 
    + i\frac {\kappa^4}4 \left(
        \frac{2}{\sigma}\BSlin + \rho (\BSlin)_\rho 
    \right)
     = 0,
\end{equation} where
\[
    \Delta_{\rho}^2 = 
    -\frac{(d-1)(d-3)}{\rho^3}\partial_{\rho}
        +\frac{(d-1)(d-3)}{\rho^2}\partial_{\rho}^2
        +\frac{2(d-1)}{\rho}\partial_{\rho}^3
        +\partial_{\rho}^4 ~ .
\]

We now use WKB to find the large $\rho$ behavior of~\eqref{eq:SC_WKB_profile}:
\begin{lem}
Let~$\BSlin(\rho)$ be a solution of~\eqref{eq:SC_WKB_profile}.
Then, 
 \[
    \BSlin\sim 
        c_1\BS_{,1}(\rho)
        +c_2\BS_{,2}(\rho)
        +c_3\BS_{,3}(\rho)
        +c_4\BS_{,4}(\rho)
        ,\qquad
    \rho\to\infty,
\]
where~$\left\{ c_i \right\}_{i=1}^4$ are complex constants,
\begin{eqnarray*} 
    \BS_{,1}(\rho) &\sim &
        \rho^{-\frac 2\sigma -i\frac{4\nu}{\kappa^4}}, \\
    \BS_{,2}(\rho) &\sim&
        \frac{1}{\rho^{\frac{2}{3\sigma}(\sigma d-1)}}
        \exp\left( 
            -i {\frac3{4\sqrt[3]{4}}} 
                (\kappa\rho)^{4/3}
                +i\frac{4\nu}{3\kappa^4} \log(\rho)
        \right), \\
    \BS_{,3}(\rho) &\sim&
        \frac{
            \exp\left( + ~
                    \frac{3\sqrt{3}}{8\sqrt[3]{4}} (\kappa\rho)^{4/3}
            \right)
        }{\rho^{\frac{2}{3\sigma}(\sigma d-1)}}
        \exp\left( 
                +i\frac{3\sqrt{3}}{8\sqrt[3]{4}} (\kappa\rho)^{4/3}
                +i\frac{4\nu}{3\kappa^4} \log(\rho)
        \right), \\
    \BS_{,4}(\rho) &\sim&
        \frac{
            \exp\left( - ~
                    \frac{3\sqrt{3}}{8\sqrt[3]{4}} (\kappa\rho)^{4/3}
            \right)
        }{\rho^{\frac{2}{3\sigma}(\sigma d-1)}}
        \exp\left( 
                +i\frac{3\sqrt{3}}{8\sqrt[3]{4}} (\kappa\rho)^{4/3}
                +i\frac{4\nu}{3\kappa^4} \log(\rho)
        \right).
\end{eqnarray*}
\end{lem}
\begin{proof}

In order to apply the WKB method, we substitute~$\BSlin(\rho) = \exp(w(\rho))$,
and expand
\[
    w(\rho) \sim w_0(\rho) + w_1(\rho) + \dots\; .
\]
Substituting~$w_0(\rho) = \alpha \rho^p$ and balancing terms shows that~$p=4/3$,
and that the equation for the leading-order, 
the~$\mathcal{O}\left( \rho^{4/3} \right)$ terms, is 
\[
    \left( w_0^\prime \right)^3 
        = \left( \frac{4\alpha}{3} \right)^3 \rho
        = i \frac {\kappa^4}4 \rho.
\] 
Therefore,
\[
    \alpha = \frac{3}{4} \sqrt[3]{i \frac{\kappa^4}4}
    =
    \frac{3\kappa^{4/3}}{4\sqrt[3]{4}}
    \cdot
    \left\{
        -i ,
        \frac{\sqrt{3}+i}2 ,
        \frac{-\sqrt{3}+i}2 
    \right\} .
\]
The equation for the next order, the~$\mathcal{O}\left( 1 \right)$ terms, is \[
    \nu +i\frac {\kappa^4}2d 
    = i\frac {\kappa^4}4 \left( \frac 2\sigma + 3\rho w_1^\prime \right), 
\] implying that \[
    w_1 = \frac 13 \left( \frac 2\sigma(1-\sigma d)
    + \frac{4\nu}{\kappa^4}i \right) \log \rho .
\]
The next-order terms are~$\mathcal{O}\left( \rho^{-4/3} \right)=o(1)$ and can be
neglected.
We therefore obtain the three solutions  $\BS_{,2}$, $\BS_{,3}$,  and $\BS_{,4}$.

    Since~\eqref{eq:SC_WKB_profile} is a fourth order ODE, another solution is
    required.
    To obtain the fourth solution, we substitute~$w_0\sim \beta\log(\rho)$  
    in~\eqref{eq:SC_WKB_profile} and obtain that the equation for the
    leading-order, the~$\mathcal{O}(1)$ terms, is \[
    -\nu +i\frac {\kappa^4}4 \left( \frac 2\sigma+\beta \right) = 0 ,
    \]
    and that the next-order terms 
    are~$\mathcal{O}\left( \rho^{-4} \right)=o(1)$ and can be
neglected.
    The fourth solution is therefore $ \BS_{,1}$.

\end{proof}

Equation~\eqref{eq:SC_WKB_profile} thus has the two algebraically-decaying
solutions, $\BS_{,1}$ and~$\BS_{,2}$, the exponentially-increasing
solution~$\BS_{,3}$, and the exponentially-decreasing solution~$\BS_{,4}$.
The fact that~$\BS_{,3}$ increases exponentially as~$\rho\to\infty$ is
inconsistent with the boundary condition~\eqref{eq:explicit_ODE_BCs}.
Therefore, 
\begin{equation}    \label{eq:noBS3}
    \BS(\rho)\sim
    c_1\BS_{,1}(\rho)
    +c_2\BS_{,2}(\rho)
    +c_4\BS_{,4}(\rho).
\end{equation}
Since~$\sigma d>4$, the exponent~$
    \frac{2}{3\sigma}(\sigma d-1)
$ of~$\BS_{,2}$ is larger that the exponent~$\frac 2\sigma$ of~$\BS_{,1}$, 
hence
\begin{equation}    \label{eq:BS1_gtr_BS2}
    \BS_{,1}\gg \BS_{,2},
    \qquad \rho\to \infty.
\end{equation}.

Direct calculations give that
\begin{subequations}    \label{eqs:SB_spaces}
    \ifx \useAPS \undefined 
        \begin{align}
               \BS_{,1}\not\in L^2( \Real^d ),
            &&   \Delta\BS_{,1}\in L^2( \Real^d ),
            &&   \BS_{,1}\in L^{2+2\sigma}( \Real^d ), \\
               \BS_{,2}\in L^2( \Real^d ),
            &&   \Delta\BS_{,2}\not\in L^2( \Real^d ),
            &&   \BS_{,2}\in L^{2+2\sigma}( \Real^d ), \\
               \BS_{,4}\in L^2( \Real^d ),
            &&   \Delta\BS_{,4}\in L^2( \Real^d ),
            &&   \BS_{,4}\in L^{2+2\sigma}( \Real^d ).
        \end{align}
    \else 
        $$
        \begin{array}{ccc} \displaystyle
               \BS_{,1}\not\in L^2( \Real^d ),
            &   \Delta\BS_{,1}\in L^2( \Real^d ),
            &   \BS_{,1}\in L^{2+2\sigma}( \Real^d ), \\
               \BS_{,2}\in L^2( \Real^d ),
            &   \Delta\BS_{,2}\not\in L^2( \Real^d ),
            &   \BS_{,2}\in L^{2+2\sigma}( \Real^d ), \\
               \BS_{,4}\in L^2( \Real^d ),
            &   \Delta\BS_{,4}\in L^2( \Real^d ),
            &   \BS_{,4}\in L^{2+2\sigma}( \Real^d ).
        \end{array}
        $$
    \fi
\end{subequations}
Therefore, $\BS$~is in~$L^{2+2\sigma}$. Unless $c_1=c_2=0$, however, 
  $\BS$~is not in~$H^2$.
Furthermore, since
\[
    \norm{\psi_{\BS}}_{2+2\sigma}^{2+2\sigma}
    = \frac1{L^{4/\sigma-(d-4)}(t)} \norm{\BS}_{2+2\sigma}^{2+2\sigma},
\]
then~$\psi_\BS\in L^{2+2\sigma}$ for~$0\le t<\TCrit$.
In the~$H^2$-subcritical regime $4/\sigma -(d-4)>0$.
Therefore, \[
    \lim_{t\to\TCrit} \norm{\psi_\BS}_{2+2\sigma} =\infty.
\]
Hence,
\begin{lem}
Assume that Conjecture~\ref{conj:explicit} holds. Then, 
$\psi_\BS$ is an explicit solution of the BNLS equation~\eqref{eq:BNLS}
that becomes singular in~$L^{2+2\sigma}$ as~$t\to\TCrit$.
\end{lem}

\subsection{Zero-Hamiltonian solutions}
\label{ssec:explicit_OH}

As is the case of peak-type solutions of the supercritical NLS, a key role is played by the zero-Hamiltonian solutions.
\begin{thrm}    \label{thrm:H_of_SB}
     Let $\sigma$  be in the $L^2$~supercritical and~$H^2$-subcritical regime~\eqref{eq:admissible-range-L+H}.
    Let~$\BS$ be a solution of~\eqref{eqs:explicit_ODE}.
    If~$H\left[ \BS \right]<\infty$, then~$H\left[ \BS \right]=0$.
    \begin{proof}
        The Hamiltonian of~$\psi_\BS$, see~\eqref{eq:psiBS}, is equal to
        \begin{equation*}
            H\left[ \psi_\BS \right]
                = \frac1{L^{4/\sigma-(d-4)}} H \left[ \BS \right].
        \end{equation*}
        From $H^2$-subcriticality, it follows
        that~$L^{-4/\sigma-4+d}(t)\neq{\it const}$.
        Therefore, Hamiltonian conservation~($
            H\left[ \psi_\BS \right]\equiv {\it const}
        $) implies that~$H\left[ \BS \right]=0$.
    \end{proof}
\end{thrm}  

\begin{lem} \label{lem:admissible_B23}
 Let $\sigma$  be in the $L^2$~supercritical and~$H^2$-subcritical regime~\eqref{eq:admissible-range-L+H}.
    Let $\BS(\rho)$ be a zero-Hamiltonian solution
    of~\eqref{eqs:explicit_ODE}.
    Then, 
\begin{enumerate}
  \item $c_2= c_3= 0$.
  \item If~$c_1\neq0$ then 
    \begin{equation}    \label{eq:BS_sim_BS1}
        \BS(\rho) \sim c_1\BS_{,1}(\rho) 
        ,\qquad
        \rho\to\infty,
    \end{equation}
    or, equivalently,
    \begin{subequations}    \label{eqs:BS_OH}
        \begin{equation}    \label{eq:BS_OH}
            \lim_{\rho\to\infty}\left( 
                \rho S^\prime + \left( 
                    \frac 2\sigma + i \frac{4\nu}{\kappa^4}
                \right)S
            \right)\rho^{\gamma}
            =0,
        \end{equation}
        where 
        \begin{equation}    \label{eq:BS_OH_gamma}
            \gamma_0<\gamma<\gamma_1,\qquad 
            \gamma_0 =\frac23\left( d-2-\frac2\sigma \right),\quad 
            \gamma_1 = 4+\frac2\sigma.
        \end{equation}
        Moreover, $\BS \in L^{2 \sigma+2}$ and~$\BS\not\in L^2$.
    \end{subequations}
\end{enumerate}
\end{lem}
\begin{proof}
    The exponentially increasing solution~$c_3\BS_{,3}$ must vanish, as
    explained above.
    Convergence of the Hamiltonian requires that~$\Delta\BS\in L^{2}$.
    Since~$\Delta\BS_{,2}\not\in L^2$, see~\eqref{eqs:SB_spaces}, it follows
    that~$c_2=0$.
    Since~$\BS_{,1}\not\in L^2$, then~$\BS\not\in L^2$.
    
    To show that~\eqref{eqs:BS_OH} is equivalent to demanding that~$c_2=0$, we
    first note that in the $H^2$-subcritical regime $d-4<\frac4\sigma$ and so
    $\gamma_0<\frac23\left(2+\frac2\sigma\right)<\gamma_1$.
    Next, direct calculation gives that \[
    \left(
        \rho \frac{d}{d\rho} +
        \frac2\sigma+i\frac{4\nu}{\kappa^4}
    \right)\BS_{,2}
    \sim {\cal O} \left(
        \rho^{4/3-\frac2{3\sigma}(\sigma d-1)}
    \right),
    \quad   \rho\to\infty,
    \]
    and that  
    \[
    \left(
        \rho \frac{d}{d\rho} +
        \frac2\sigma+i\frac{4\nu}{\kappa^4}
    \right)\BS_{,1}
    \sim {\cal O}\left(
        \rho^{-\frac{2}{\sigma}-4}
    \right),
    \quad\rho\to\infty,
    \]
    where the LHS is the result of the next term in the WKB approximation
    of $\BS_{,1}$.
    Therefore, \[
        \rho^\gamma
        \left(
            \rho \frac{d}{d\rho} +
         \frac2\sigma+i\frac{4\nu}{\kappa^4}
        \right)\BS
        \sim {\cal O} \left(
            c_1\cdot\rho^{\gamma-\frac{2}{\sigma}-4}
        \right)
        +
        {\cal O}\left(
            c_2\cdot\rho^{\gamma+4/3-\frac{2}{3\sigma}(\sigma d-1)}
        \right),
        \quad \rho\to\infty.
    \]
    Since, for $\gamma_0<\gamma<\gamma_1$, \[
        \lim_{\rho\to\infty}
            \rho^{\gamma-\frac{2}{\sigma}-4}=0,
        \qquad
        \lim_{\rho\to\infty}
            \rho^{\gamma+\frac43-\frac2{3\sigma}(\sigma d-1)}=\infty,
    \]
    it follows that $c_2=0$ if and only if the limit~\eqref{eq:BS_OH} is
    satisfied.
\end{proof}

The fourth-order nonlinear ODE~\eqref{eq:explicit_ODE_profile} requires 
four boundary conditions.
Three boundary conditions are given by~\eqref{eq:explicit_ODE_BCs}, and 
the fourth condition will be the  zero-Hamiltonian condition~\eqref{eq:BS_OH}.  
Generically, one can expect that for a given~$\nu$, this nonlinear eigenvalue problem has an enumerable
number of eigenvalues~$\kappa^{(n)}$ with  corresponding eigenfunctions~$\BS^{(n)}$.
As in the case of the supercritical NLS~\cite{Koppel-95,Budd-01}, we
conjecture that for any~$(\sigma,d,\nu)$ there is a unique {\em admissible
solution}, which is monotonically decreasing.

\begin{conj}
    Let $\sigma$  be in the $L^2$~supercritical and~$H^2$-subcritical
    regime~\eqref{eq:admissible-range-L+H}, and let~$\nu>0$.
   Then, the nonlinear eigenvalue problem posed by
    equation~\eqref{eq:explicit_ODE_profile}, subject to the boundary
    conditions
    \begin{equation}    \label{eq:explicit_ODE_BCs_OH}
        \BS^\prime(0)=\BS^{\prime\prime\prime}(0)=\BS(\infty) =0,
        \qquad 
        \lim_{\rho\to\infty}\left( 
            \rho S^\prime + \left( 
            \frac 2\sigma + i \frac{4\nu}{\kappa^4}
            \right)S
        \right)\rho^{\gamma}
        =0,
    \end{equation}
    where~$\gamma$ satisfies~\eqref{eq:BS_OH_gamma}, admits a unique
    eigenpair~$(\BS^{\rm admis.}(\rho),\kappa^{\rm admis.})$, such that
$$
\kappa^{\rm admis.} = \kappa^{\rm admis.}(\sigma, d,\nu)>0,
$$
and $|\BS^{\rm admis.}(\rho)|$ is monotonically-decreasing.
Furthermore, 
$$
    \BS^{\rm admis.}(\rho) \sim 
        c \rho^{
            -2/\sigma
            -4\nu i/\kappa^4
            }, \qquad 
        \rho \longrightarrow  \infty,
$$
and 
$$
 \kappa^{\rm admis.}(\sigma, d,\nu)  =\nu^{1/4} \tilde{\kappa}^{\rm admis.}(\sigma, d), \qquad 
\tilde{\kappa}^{\rm admis.}:=\kappa^{\rm admis.}(\sigma, d,\nu=1).
$$
\end{conj}

\section{\label{sec:asymptotic}Peak-type $H^2$-singular solutions}

\subsection{\label{ssec:asymptotic_analysis}Informal analysis}

As in the supercritical NLS ~\cite{LeMesurier-88a,Shvets-93},   
we expect that singular peak-type solutions of the
supercritical BNLS undergo a quasi self-similar collapse, so that 
\begin{equation}    \label{eq:SC_quasi_parts}
    \psi(t,r) \sim
    \begin{cases}
        \psi_\BS(t,r)    &    0\le r\le r_c, \\
        \psi_{\text{non-singular}}(t,r)\quad    &    r\ge r_c ,
    \end{cases}
\end{equation}
where~$\psi_\BS$ is the self-similar profile~\eqref{eq:psi_SB}.
The singular region~$r\in[0,r_c]$ is constant in the coordinate~$r$. 
Therefore, in the rescaled variable~$\rho=r/L(t)$, the singular
region~$\rho\in[0,r_c/L(t)]$ becomes infinite as~$L(t)\to0$.
This is in contradistinction with the critical-BNLS case, where the singular
region~$\rho\in[0,\rho_c]$ is constant in the rescaled variable~$\rho$, but
shrinks to a point in the original coordinate~$r$~\cite{Baruch_Fibich_Mandelbaum:2009a}.

\begin{lem} \label{lem:supercrit_peak_rate_profile}
    Let~$\sigma d>4$, and let~$\psi$ be a peak-type singular solution of the
    BNLS that collapses with the~$\psi_\BS$
    profile~\eqref{eq:psi_SB}.
    If~$L(t)\sim\kappa(\TCrit-t)^p$, then~$p \ge \frac14$.
    Furthermore,
\begin{itemize}
  \item If~$p=1/4$ if then the self-similar profile~$\BS(\rho)$
    satisfies the equation~\eqref{eq:explicit_ODE_profile}.
 \item If~$p>1/4$,
    then the profile satisfies the equation 
\begin{equation}    \label{eq:stationary_state}
    -\Delta^2 \BR(\bvec{x}) - \BR + |\BR|^{2\sigma}\BR = 0.
\end{equation}
 \end{itemize}
   \begin{proof}
        If~$\psi\sim\psi_\BS$, then the equation for~$\BS$ is
        \begin{equation}    \label{eq:peak_SCBNLS_ODE_1}
            -\nu\BS - i \left(
                \lim_{t\to\TCrit} L_tL^3
            \right) \left(
                \frac 2\sigma \BS + \rho \BS^\prime 
            \right)
            -\Delta^2_\rho \BS +\abs{\BS}^{2\sigma}\BS = 0,
        \end{equation}
        implying that~$L_t L^3$ should be bounded as~$t\to \TCrit$.
        Since~$L^3L_t \sim - p\kappa^4(\TCrit-t)^{4p-1}$, it follows
        that~$p\ge\frac14$.
        If~$p=1/4$, then equation~\eqref{eq:peak_SCBNLS_ODE_1} reduces
        to~\eqref{eq:explicit_ODE_profile}, see Section~\ref{sec:explicit}.
    \end{proof}
\end{lem} 

From Hamiltonian conservation it follows that~$
    H\left[ \psi_\BS \right]
$ is bounded, because otherwise the non-singular region would also have an
infinite Hamiltonian.
Therefore, from Theorem~\ref{thrm:H_of_SB} it follows 
that~$H\left[ \BS \right]=0$.

In Lemma~\ref{lem:admissible_B23} we saw that the zero-Hamiltonian solutions
of~\eqref{eq:explicit_ODE_profile} are in~$L^{2+2\sigma}$, but not in~$L^2$. Hence,  $\psi_\BS\not\in L^2$. From power conservation, however,
 it follows that if~$\psi_0\in H^2$, 
then~$\psi\in L^2$. As in the NLS case, see~\cite{Berge-92}, 
this ``contradiction'' can be resolved as follows.
\begin{cor}    \label{cor:admissible_B23_power}
    Let $\BS(\rho)$ be a zero-Hamiltonian solution
    of~\eqref{eqs:explicit_ODE}.
    Then, $\norm{\BS}_2=\infty$.
    Nevertheless,~$
        \displaystyle \lim_{t\to\TCrit}
            \norm{\psi_\BS}_{L^2(r<r_c)}<\infty
    $.
\end{cor}
\begin{proof}
    Since $\BS(\rho)\sim c_1\BS_{,1}(\rho)$, 
    \[
        \norm{\BS_{,1}}_2^2 
        \sim C \int_{\rho=0}^{\infty}
            \rho^{-4/\sigma+d-1} d\rho 
            \sim C \rho^{d-4/\sigma} \Big|_{\rho=0}^{\infty}
            = \infty.
    \]
    The profile~$\psi_\BS$ satisfies
    \begin{eqnarray*}
        \norm{\psi_\BS}_{L^2(r<r_c)}^2 
        &=&
        L^{d-4/\sigma}(t) \cdot 
        \int_{\rho=0}^{r_c/L(t)}
            \abs{\BS(\rho)}^2
            \rho^{d-1}d\rho \\
        &\sim& 
        L^{d-4/\sigma}(t) \cdot \left( 
            C \rho^{d-4/\sigma} \Big|_{\rho=0}^{r_c/L(t)}
        \right) 
        =
        \mathcal{O}(1).
    \end{eqnarray*}
\end{proof}

In summary, we conjecture the following:
\begin{conj}    \label{conj:supercrit_peak_rate_profile} ~\\
    Let $\psi$ be peak-type singular solution of the supercritical BNLS. 
    Then,
    \begin{subequations} \label{eqs:supercrit_peak_QSS}
        \begin{enumerate}
            \item The collapsing core approaches the self-similar
                profile~$\psi_\BS$, i.e.,
                \begin{equation}    \label{eq:supercrit_peak_QSS-1}
                    \psi(t,r) \sim \psi_\BS(t,r), \qquad 
                    0\le r\le r_c,
                \end{equation}
                where
                \begin{equation}    \label{eq:supercrit_peak_QSS-2} 
                    \psi_\BS(t,r) = 
                        \frac1{L^{2/\sigma}(t)}    \BS(\rho)    e^{i \nu \tau(t)},\qquad 
                    \rho = \frac rL,\qquad 
                    \tau(t) = \int_{s=0}^{t}\frac{1}{L^{4}(s)}ds.
                \end{equation}
            \item The self-similar profile~$\BS(\rho) = \BS^{\rm admis.}(\rho)$ is the unique admissible solution of 
                \begin{equation}    \label{eq:supercrit_peak_ODE}
                    \begin{gathered}
                        - \nu \BS(\rho) + i\frac{\kappa^4}4 \left(
                            \frac{2}{\sigma}\BS + \rho \BS^\prime 
                        \right)
                        - \Delta_\rho^2 \BS + |\BS|^{2\sigma}\BS = 0, \\
                        \BS^\prime(0)=\BS^{\prime\prime\prime}(0)=0,
                        \qquad
                        \BS(\infty)=0,
                        \qquad
                        H[\BS]=0,
                    \end{gathered}
                \end{equation}
                where $\kappa^{\rm admis.}(\sigma, d,\nu)  =\nu^{1/4} \tilde{\kappa}^{\rm admis.}(\sigma, d) >0$.

            \item In particular, $\BS(\rho)\neq \BR(\rho)$.
 
            \item The blowup rate of singular peak-type solutions is exactly a
                    quartic root, i.e.,\begin{equation}    \label{eq:rate_14}
                    L(t)\sim\kappa\sqrt[4]{\TCrit-t},
                    \qquad \kappa>0.
                \end{equation}
            \item The coefficient~$\kappa$ of the blowup rate of~$L(t)$ is equal
                    to the value of~$\kappa$ of the admissible solution~$\BS$,
                    i.e., \[
                    \kappa := \lim_{t\to\TCrit} 
                        \frac{L(t)}{\sqrt[4]{\TCrit-t}}  
                    = \kappa^{\rm admis.}(\sigma,d,\nu).
                \] 
                In particular, $\kappa$  is universal (i.e., it does not depend
                on the initial condition).
        \end{enumerate}
    \end{subequations}
\end{conj}

\noindent In Section~\ref{ssec:simulations} we provide numerical
evidence in support of Conjecture~\ref{conj:supercrit_peak_rate_profile}.

\subsection{Simulations}
\label{ssec:simulations}

\begin{figure}
    \centering
    \subfloat[$d=1, \sigma=6$, $\psi_0 = 1.6e^{-x^2}$]{\label{fig:supercrit_peak_amp_Ls_1D}%
        \includegraphics[angle=-90,clip,width=0.35\textwidth]%
            {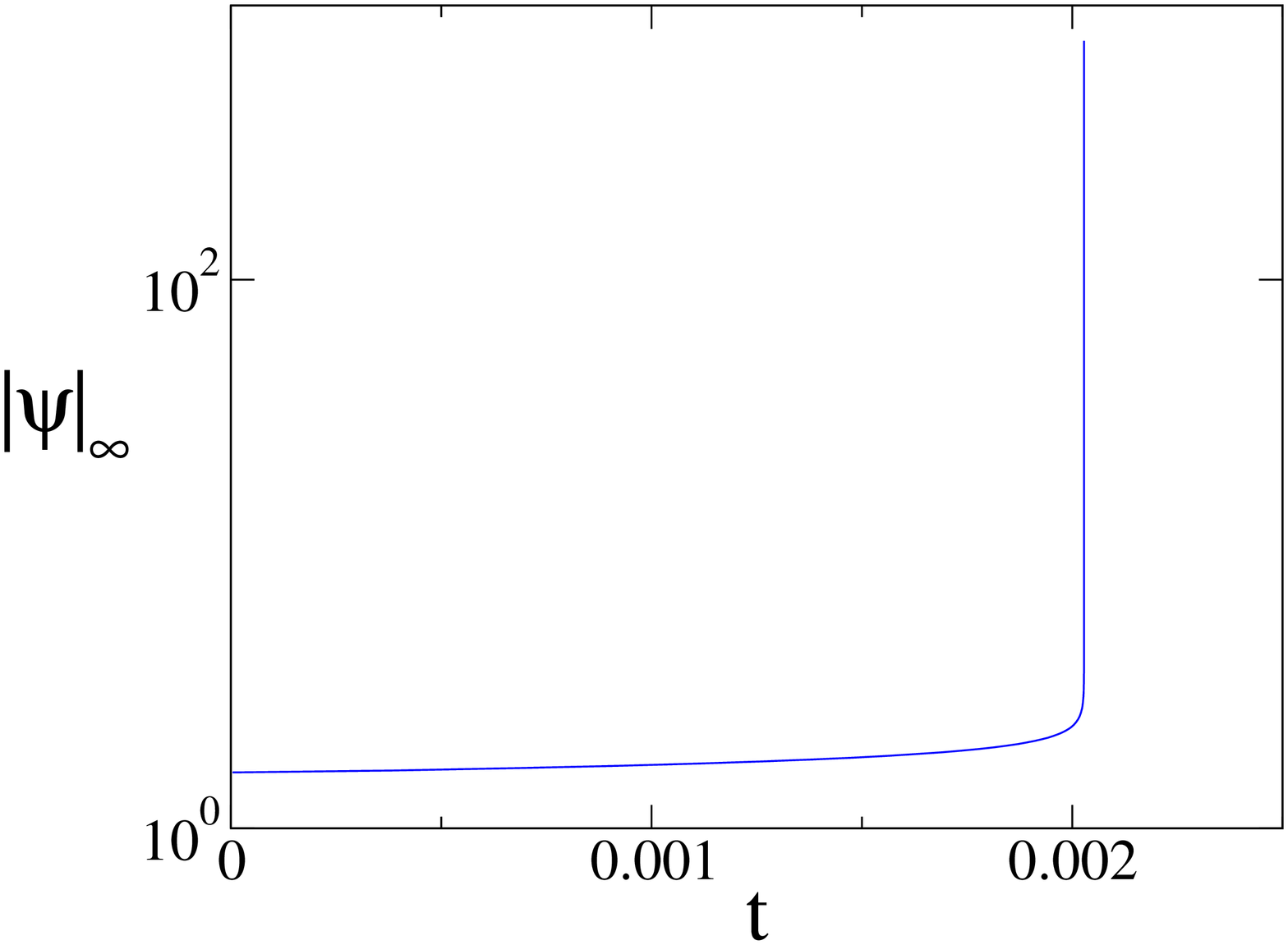}%
    }%
    \subfloat[$d=2, \sigma=3$, $\psi_0 = 3e^{-r^2}$]{\label{fig:supercrit_peak_amp_Ls_2D}%
        \includegraphics[angle=-90,clip,width=0.35\textwidth]%
            {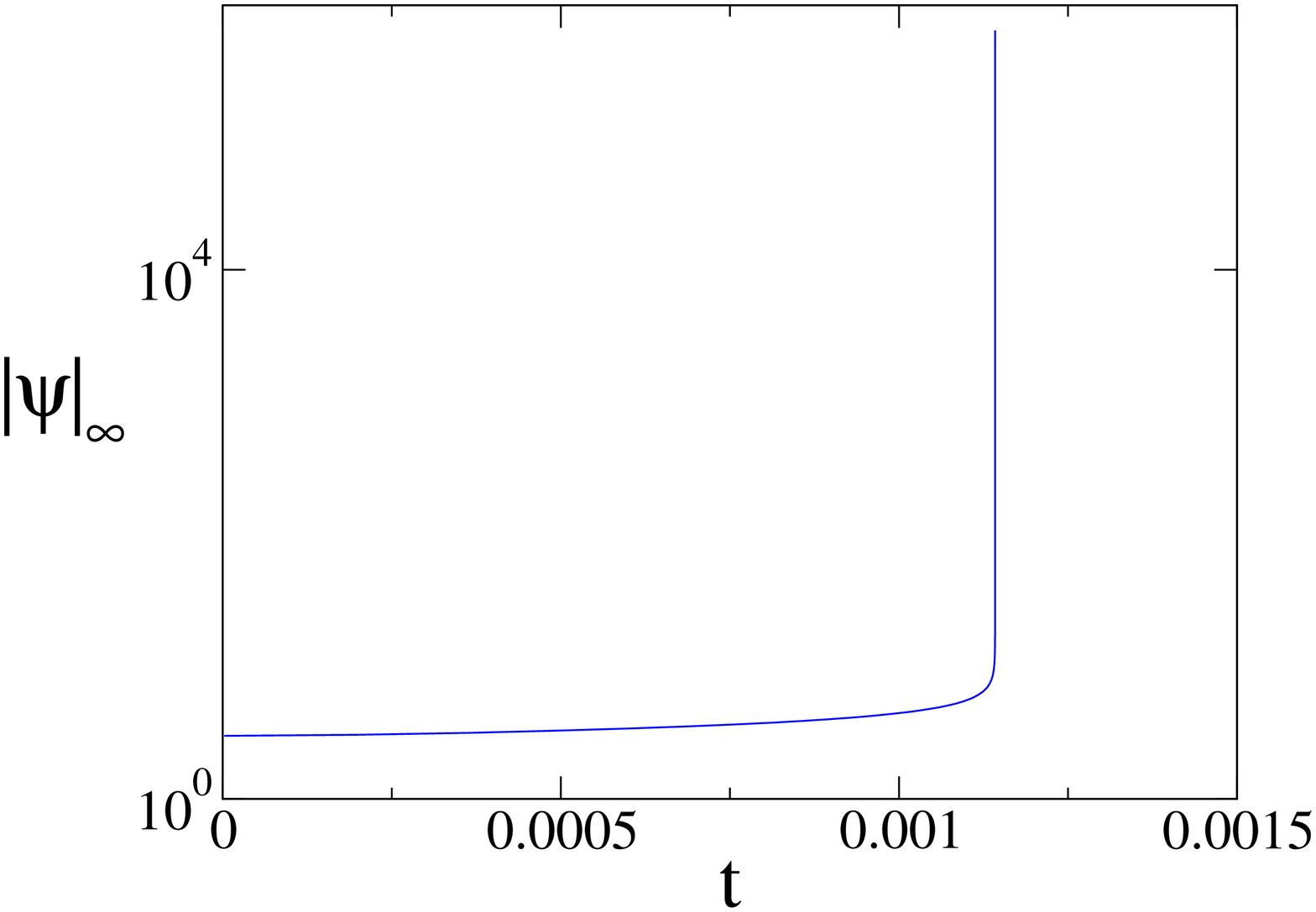}%
    }

    \mycaption{\label{fig:supercrit_peak_amp_Ls}%
        Maximal amplitude of singular solutions of the supercritical BNLS.
    }
\end{figure}
The radially-symmetric BNLS~\eqref{eq:radial_BNLS} was solved in the
supercritical cases:
\begin{enumerate}
  \item $d=1,\sigma=6$ with the initial 
condition~$\psi_0(x)=1.6e^{-x^2}$.
\item $d=2,\sigma=3$ with the initial condition~$\psi_0(r)=3e^{-r^2}$.
\end{enumerate}
In both cases, the solutions blowup at a finite time, see
Figure~\ref{fig:supercrit_peak_amp_Ls}.

\begin{figure}
    \centering
    \subfloat[$d=1,\sigma=6$]{\label{fig:supercrit_peak_rescaled_1D}%
        \includegraphics[clip,angle=-90,width=0.45\textwidth]%
            {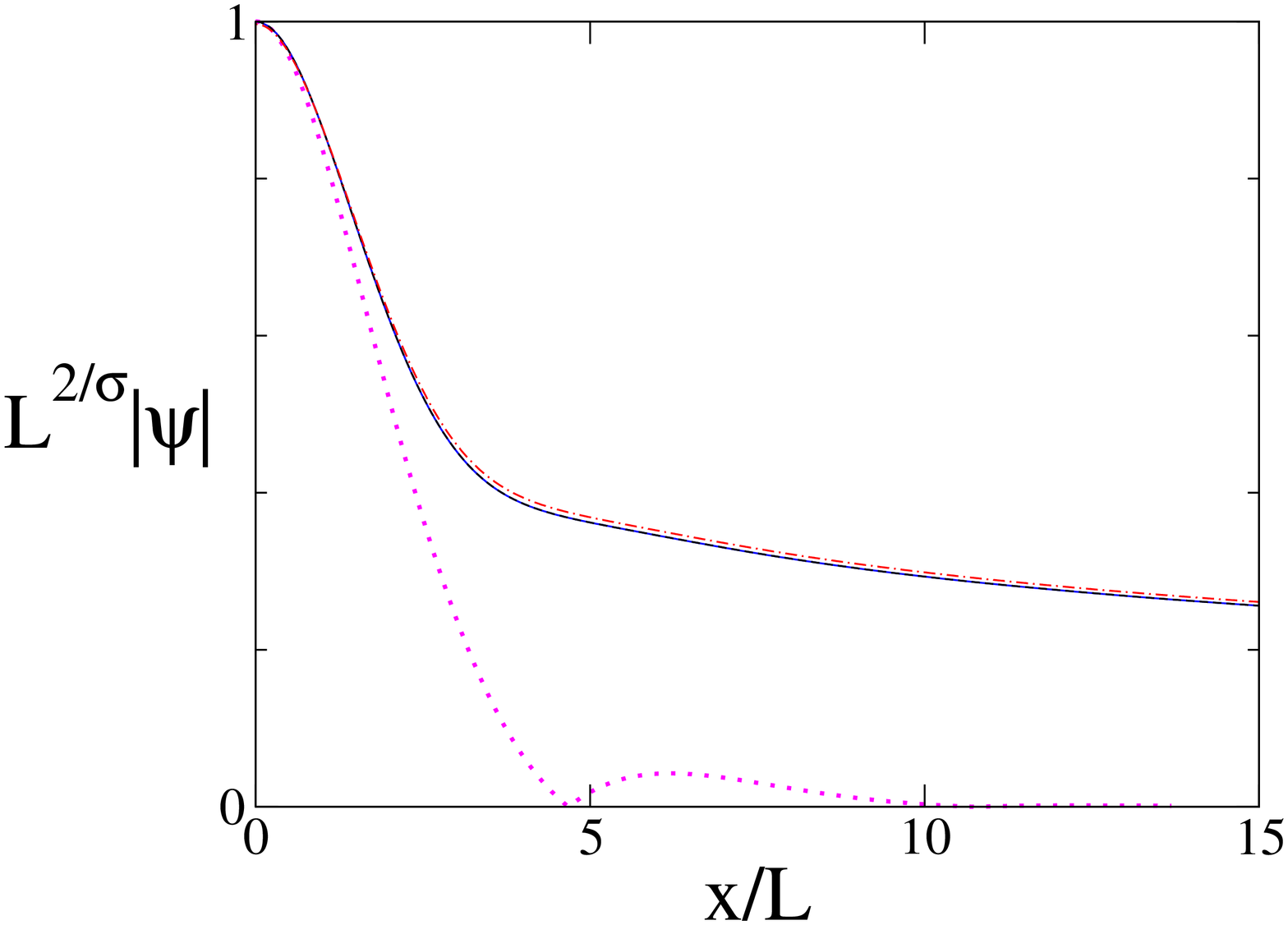}%
    }%
    \subfloat[$d=2,\sigma=3$]{\label{fig:supercrit_peak_rescaled_2D}%
        \includegraphics[clip,angle=-90,width=0.45\textwidth]%
            {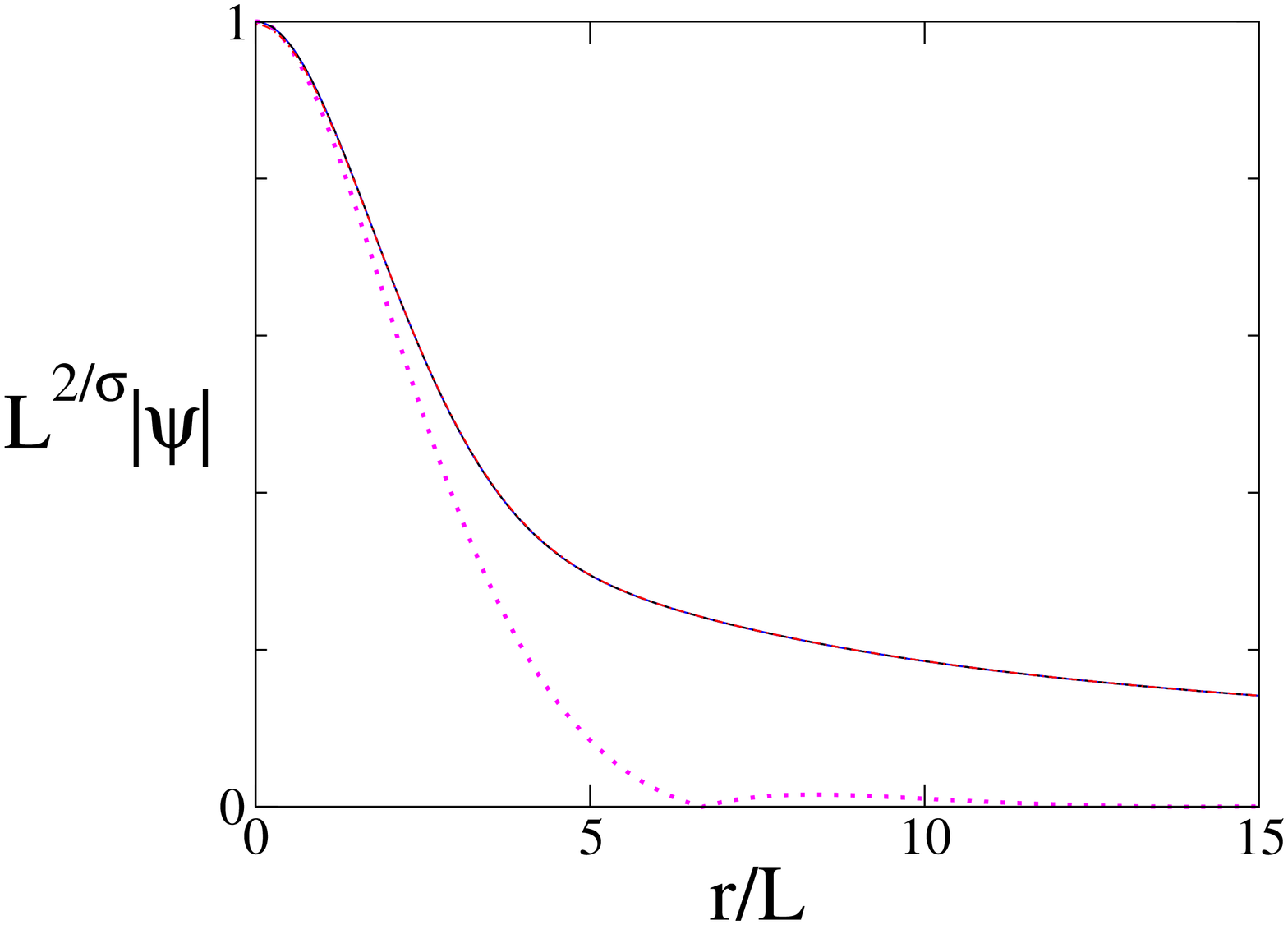}%
    }%

    \mycaption{\label{fig:supercrit_peak_rescaled}
        The solutions of Figure~\ref{fig:supercrit_peak_amp_Ls}, rescaled
        according to~\eqref{eq:psi_rescaled_peak}, at the focusing levels~$1/L=10^4$
        (blue solid line) and~$1/L=10^8$ (black dashed line).
        The red dash-dotted line is the rescaled solution~$\BS$ of~\eqref{eqs:supercrit_peak_QSS}.
        The magenta dotted line is the rescaled ground-state~$R$.
    }
\end{figure}
\begin{figure}
    \centering
    \subfloat[$d=1,\sigma=6$]{\label{fig:supercrit_peak_rescaled_far_1D}%
        \includegraphics[clip,angle=-90,width=0.45\textwidth]%
            {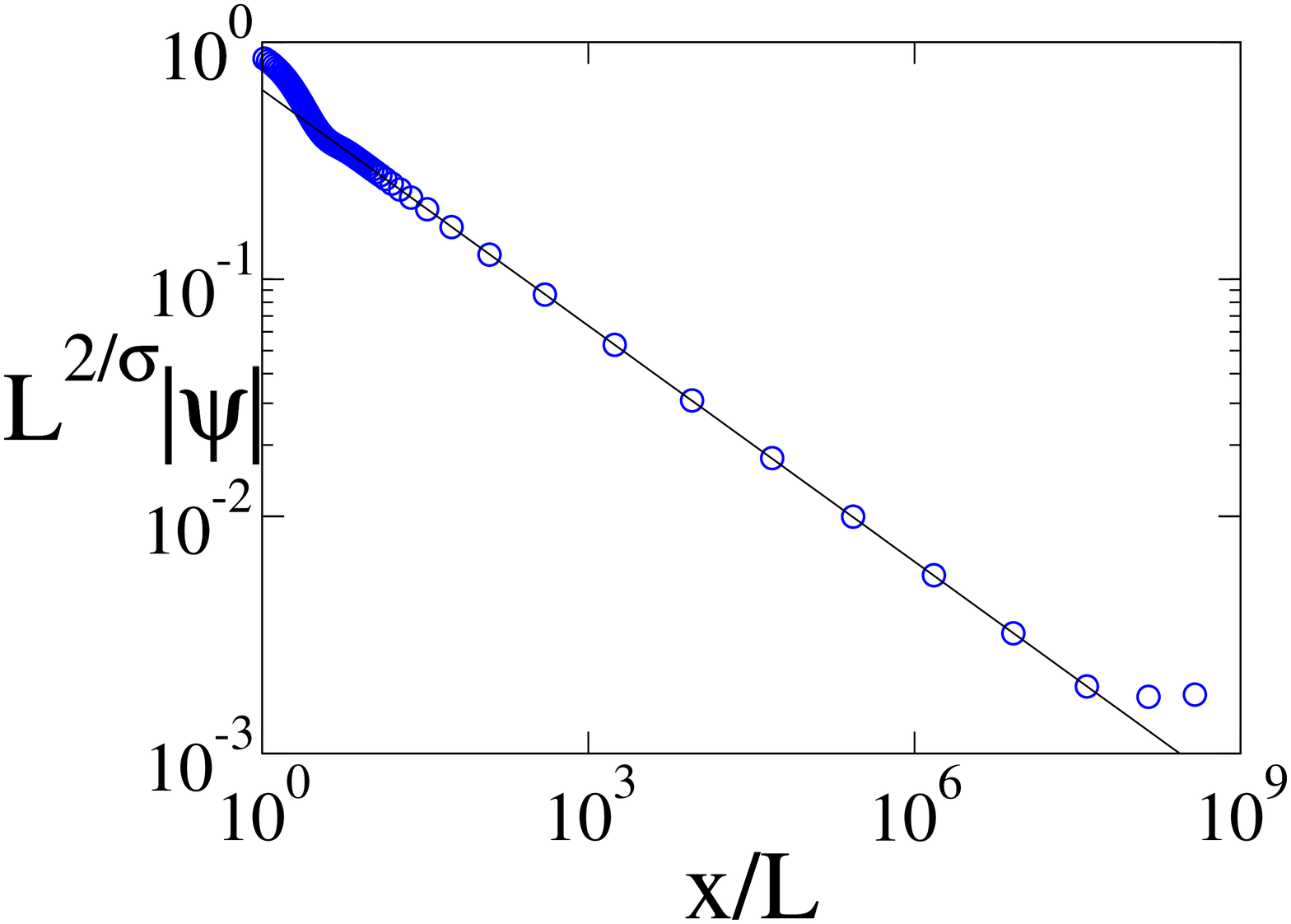}%
    }%
    \subfloat[$d=2,\sigma=3$]{\label{fig:supercrit_peak_rescaled_far_2D}%
        \includegraphics[clip,angle=-90,width=0.45\textwidth]%
            {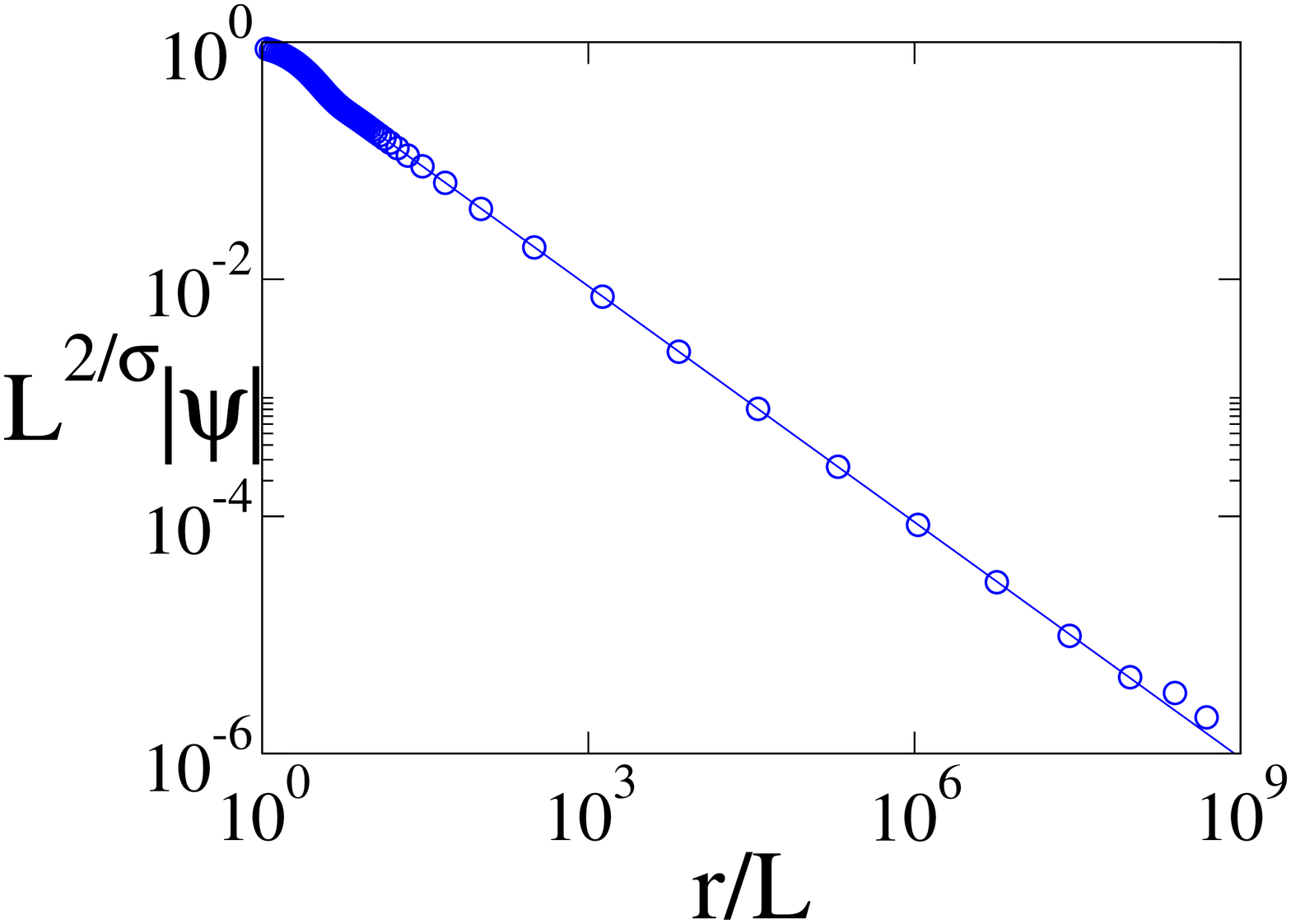}%
    }%

    \mycaption{\label{fig:supercrit_peak_rescaled_far}
        The solutions of Figure~\ref{fig:supercrit_peak_amp_Ls}, rescaled
        according to~\eqref{eq:psi_rescaled_peak}, at the focusing level~$1/L=10^8$
        (circles).
        Solid lines are the fitted curves~$y=0.63\cdot( x/L )^{-0.33}$ (left)
        and~$y=0.85\cdot( r/L )^{-0.66}$ (right).
    }
\end{figure}
To check whether the solutions collapse with the self-similar
profile~\eqref{eqs:supercrit_peak_QSS}, the solution was rescaled according to
\begin{equation}    \label{eq:psi_rescaled_peak}
    \psi_{\text{rescaled}}(t,\rho) =
        L^{2/\sigma}(t) \psi(t,r=\rho\cdot L), \qquad
    L(t)=\norm{\psi}_\infty^{-2/\sigma}.
\end{equation}
Comparing this rescaling with~\eqref{eq:supercrit_peak_QSS-2} shows that it implies
that $|\BS(0)|=\norm{\BS}_\infty  = 1$. 
This requirement  can always be satisfied with a proper choice of~$\nu$,
see~\eqref{eq:Sb_tildeSb}.
Figure~\ref{fig:supercrit_peak_rescaled_1D} shows the rescaled solutions at the
focusing levels~$L=10^{-4}$ and~$L=10^{-8}$, and the rescaled solution~$\BS(\rho)$
of~\eqref{eqs:supercrit_peak_QSS}.\footnote{
    In the calculation of~$\BS$, see \myappendix\ref{sec:BS_numerics}, the values
    of~$\nu$ and~$\kappa$ were extracted from the BNLS simulation as discussed
    below, see equations~(\ref{eq:kappa_numeric},\ref{eq:nu_numeric}).
}
The three curves are indistinguishable, showing that the solution is
self-similar with the~$\BS$ profile, and not with the $\BR$~profile.
As additional evidence, Figure~\ref{fig:supercrit_peak_rescaled_far} shows that
as~$\rho\to\infty$, the self-similar profile of~$\psi$ decays
as~$\rho^{-2/\sigma}$, which is in agreement with the decay rate
of~$\BS_{,1}(\rho)$.

\begin{figure}
\centering
    \subfloat[$d=1,\sigma=6$]{%
        \includegraphics[angle=-90,clip,width=0.45\textwidth]%
            {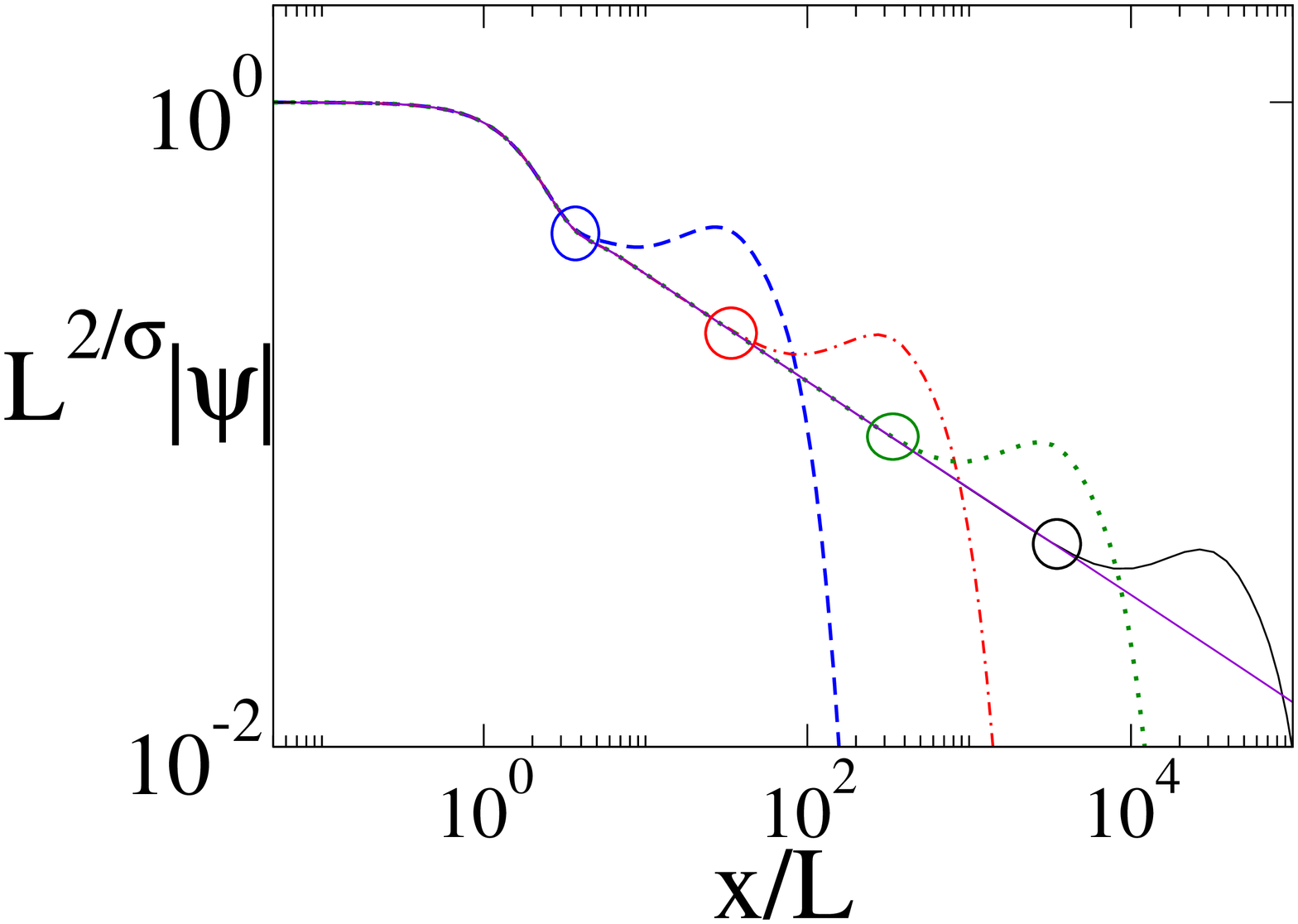}%
    }
    \subfloat[$d=2,\sigma=3$]{%
        \includegraphics[angle=-90,clip,width=0.45\textwidth]%
            {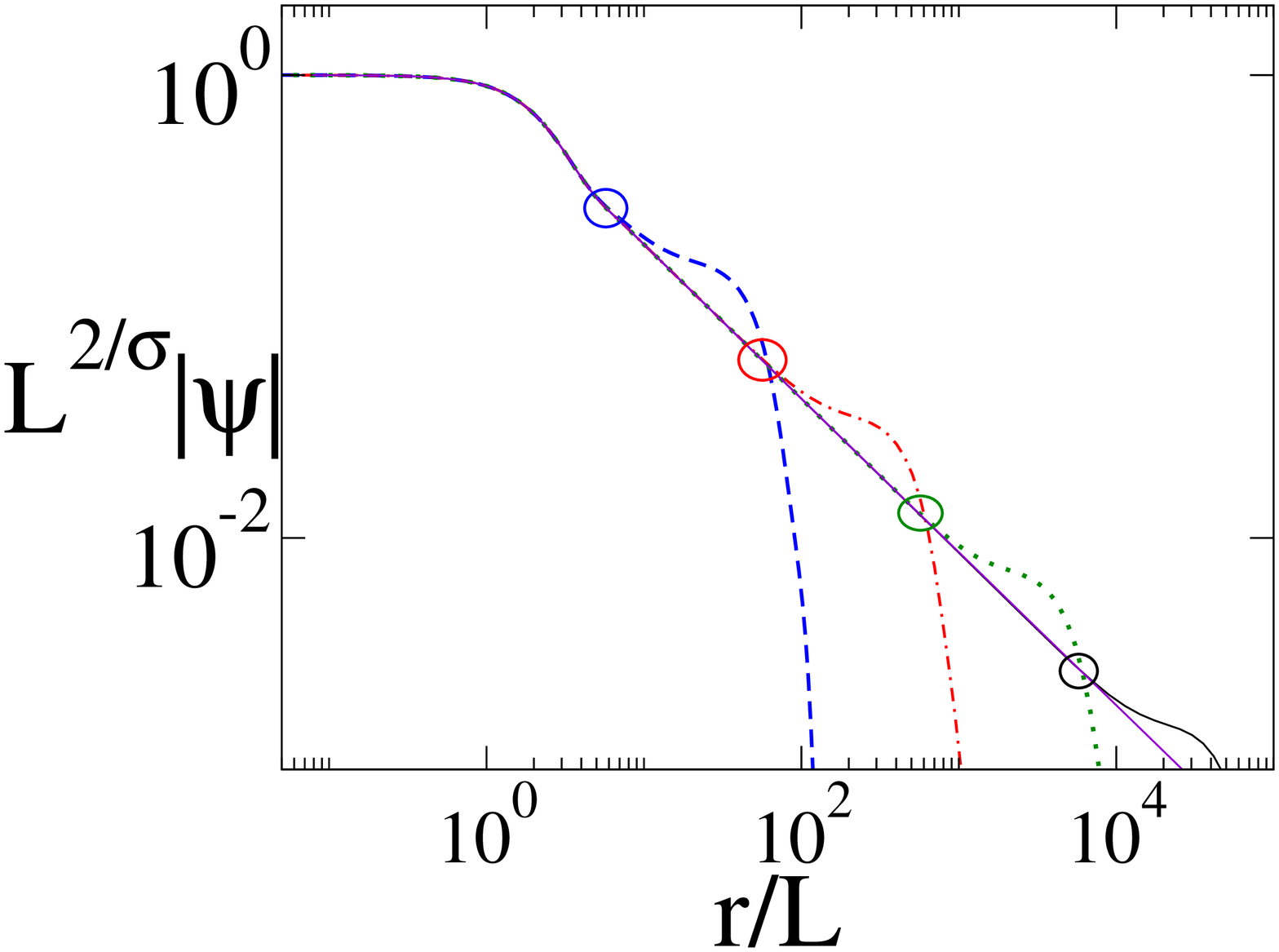}%
    }

    \mycaption{\label{fig:supercrit_peak_convergence}
        Convergence to a self-similar profile. 
        The solutions of Figure~\ref{fig:supercrit_peak_amp_Ls}, rescaled
        according to~\eqref{eq:psi_rescaled_peak}, as a function
        of~$\log(r/L)$, at the focusing levels  
            $L=10^{-1}$ (dashed blue line),
            $L=10^{-2}$ (dash-doted red line),
            $L=10^{-3}$ (dotted green line),
            $L=10^{-4}$ (solid black line)
            and 
            $L=10^{-8}$ (solid magenta line).
        The circles mark the approximate position where each curve bifurcates
        from the limiting profile, see also
        Table~\ref{tab:supercrit_peak_convergence}.
    }
\end{figure}
\begin{table}
    \centering
    \begin{tabular}{|c||c|c|c|c||c|} 
        \hline
        $1/L$ & $10$ & $100$ & $1000$ & $10000$ & $r_c$ \\
        \hline
        \hline
        $x/L~(d=1)$ &
            $3.6$ & $36$ & $360$ & $3600$ & $0.36$ \\
        \hline
        $r/L~(d=2)$ &
            $6$ & $60$ & $600$ & $6000$ & $0.6$\\
        \hline
    \end{tabular}

    \mycaption{\label{tab:supercrit_peak_convergence}%
        Position of circles in Figure~\ref{fig:supercrit_peak_rescaled_far}.
    }
\end{table}
We next verify that the solution converges to the asymptotic profile
for~$r\in[0,r_c]$, i.e., for~$\rho\in[0,r_c/L(t)]$.
To do this, we plot in Figure~\ref{fig:supercrit_peak_convergence} the rescaled 
solution at focusing levels of~$1/L=10,100,1000,10000$, as a function
of~$\log(r/L)$.
The curves are indistinguishable at~$r/L=\mathcal{O}(1)$, but bifurcate at
increasing values of~$r/L$.
These ``bifurcations positions'' are marked by circles in
Figure~\ref{fig:supercrit_peak_convergence}, and their $r/L$ values are listed in
Table~\ref{tab:supercrit_peak_convergence}.
The ``bifurcation positions'' are linear in~$1/L$, indicating that the region
where~$\psi\sim\psi_\BS$ is indeed~$\rho\in[0,r_c/L(t)]$, which corresponds
to~$r\in[0,r_c]$.

\begin{figure}
\centering
    \subfloat[$d=1,\sigma=6$]{%
        \includegraphics[clip,width=0.35\textwidth,angle=-90]%
            {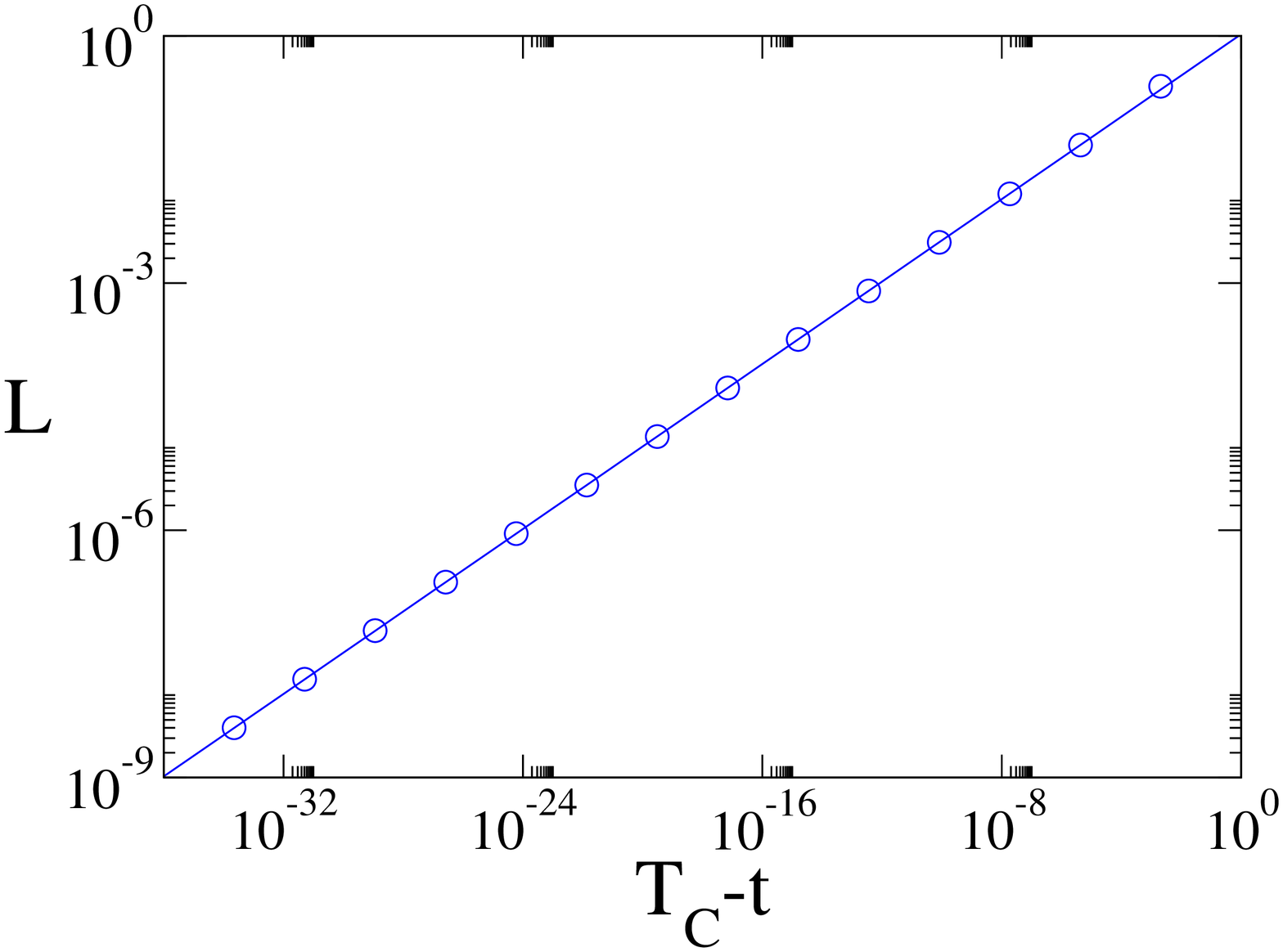}
    }
    \subfloat[$d=2,\sigma=3$]{%
        \includegraphics[clip,width=0.35\textwidth,angle=-90]%
            {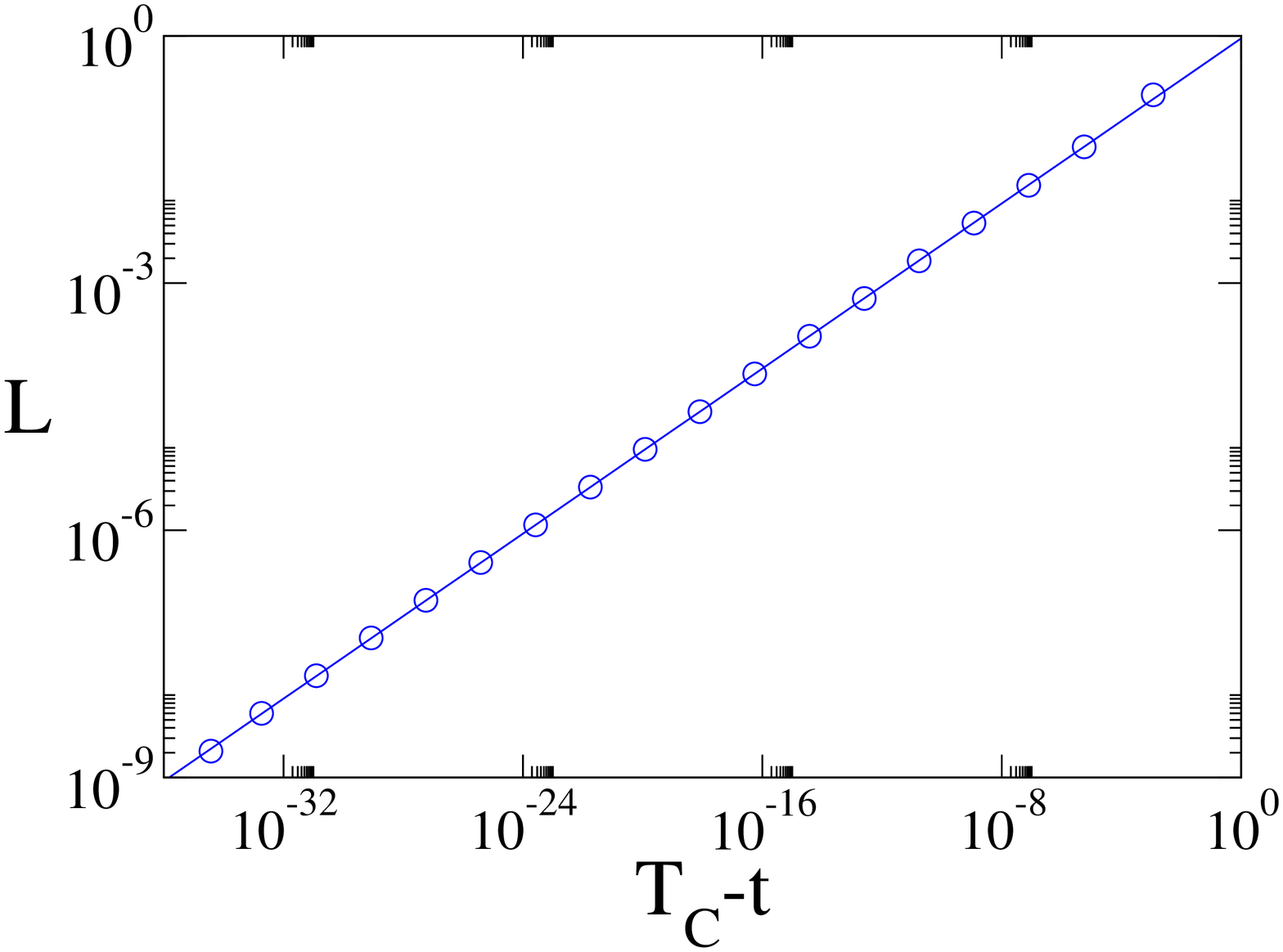}
    }

    \mycaption{\label{fig:supercrit_peak_powerlaw}
        $L(t)$ as a function of~$\left.(\TCrit-t)\right.$, on a logarithmic
        scale, for the solutions of Figure~\ref{fig:supercrit_peak_amp_Ls}
        (circles).
        Solid lines are the fitted
        curves~$\left.L=1.048\cdot(\TCrit-t)^{0.2502}\right.$ (A) 
        and~$\left.L=0.931\cdot(\TCrit-t)^{0.2504}\right.$ (B).
    }
\end{figure}
\begin{figure}
    \begin{center}
    \subfloat[]{%
        \label{fig:peak_L3Lt}%
        \includegraphics[clip,width=0.35\textwidth,angle=-90]%
            {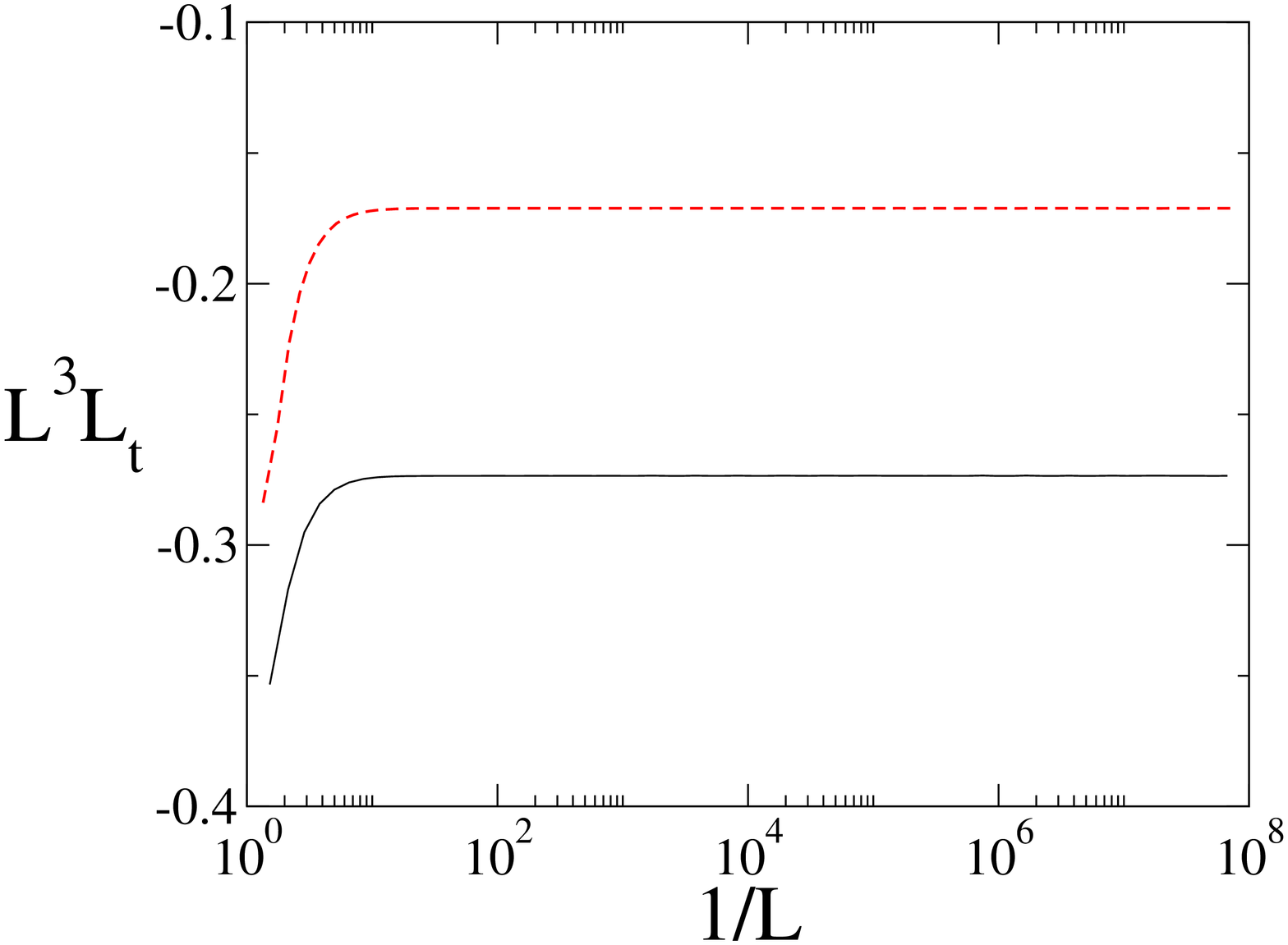}%
    }
    \subfloat[]{%
        \label{fig:peak_L4phit}%
        \includegraphics[clip,width=0.35\textwidth,angle=-90]%
            {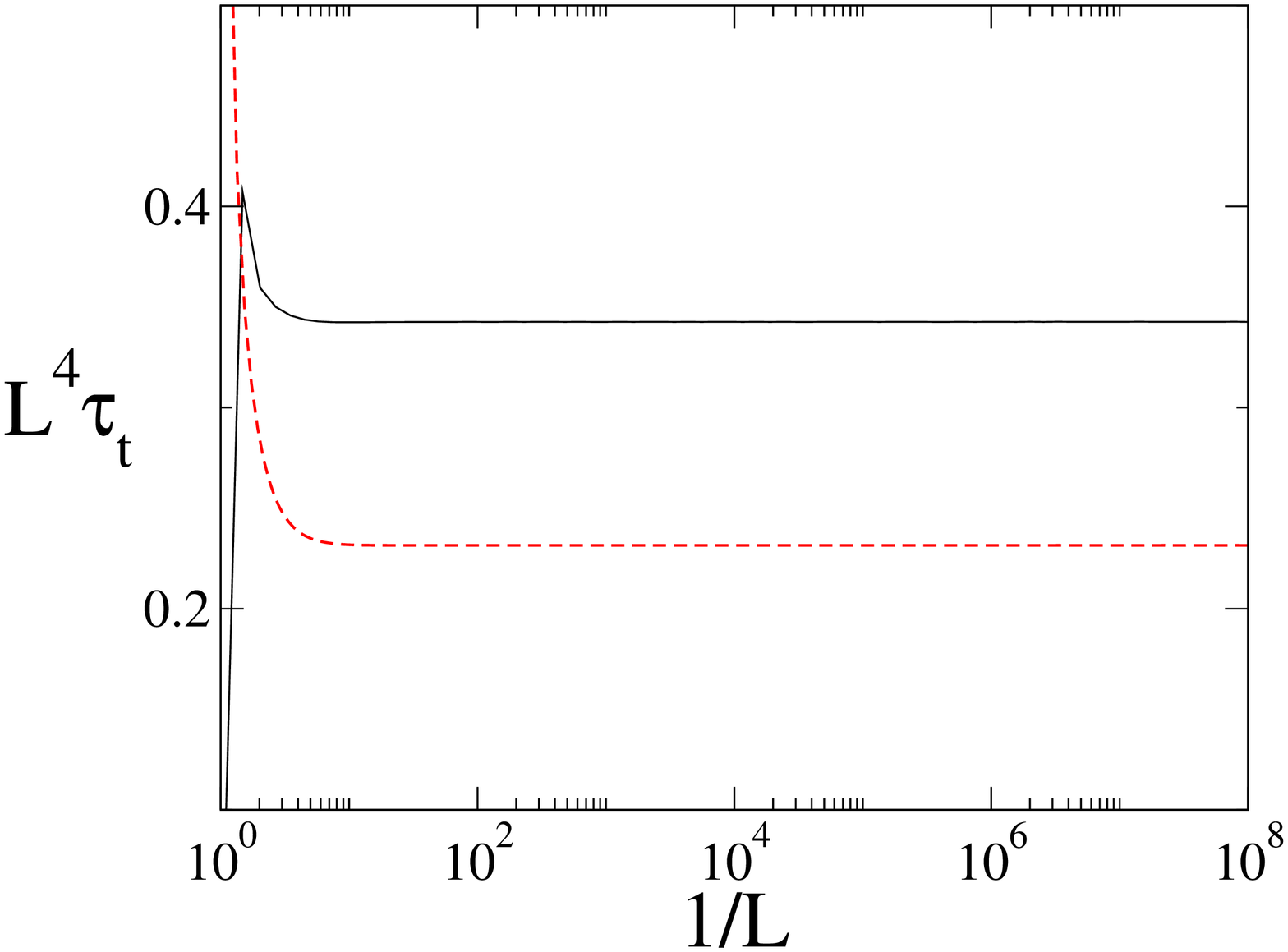}%
    }
    \end{center}
    \mycaption{
        A:~$L^3L_t$ as a function of~$1/L$, for the solution of
        Figure~\ref{fig:supercrit_peak_amp_Ls_1D} (black solid line) and 
        of Figure~\ref{fig:supercrit_peak_amp_Ls_2D} (red dashed line).
        B:~Same as~(A) for $L^4\tau_t$, where~$\tau=\arg\psi(t,r=0)$.
    }
\end{figure}
In order to compute the blowup rate~$p$,  we performed a least-squares fit of~$\log(L)$
with $\log(T_c-t)$, see Figure~\ref{fig:supercrit_peak_powerlaw}.
The resulting values are~$p\approx 0.2502$ in the~$d=1,\sigma=6$ case, 
and~$p\approx 0.2504$ in the~$d=2,\sigma=3$ case.
Next, we provide two indications that the blowup rate is exactly~$1/4$, i.e.,
that
\[
    L(t)\sim \kappa\sqrt[4]{\TCrit-t},\qquad \kappa>0.
\]
\begin{enumerate}
  \item If the blowup rate is exactly a quartic root, then~$
    L^3L_t\to -\frac{\kappa^4}{4}<0.
$
Indeed, Figure~\ref{fig:peak_L3Lt} shows that 
in the case~$d=1$,~$\sigma=6$,~$L^3L_t\to -0.289$, implying that 
\begin{subequations}
\label{eq:kappa_numeric}
\begin{equation}    \label{eq:kappa_d1s6}
    \kappa^{\rm admis.}(d=1,\sigma=6) \approx \sqrt[4]{4\cdot0.289} \approx 1.037\,.
\end{equation}
In the case~$d=2$,~$\sigma=3$,~$L^3L_t\to -0.171$, implying that 
\begin{equation}    \label{eq:kappa_d2s3}
    \kappa^{\rm admis.}(d=2,\sigma=3) \approx \sqrt[4]{4\cdot0.171} \approx 0.909\,.
\end{equation}
\end{subequations}
Since~$L^3L_t$ converges to a finite, negative constant, this shows that the
blowup rate is exactly~$1/4$.

 \item According to Lemma~\ref{lem:supercrit_peak_rate_profile},
if $\lim_{t \to T_c} L^3L_t< 0$, the self-similar profile~$\BS(\rho)$ does not satisfy the
standing-wave equation~\eqref{eq:stationary_state}, but rather is a
solution to the problem~\eqref{eq:supercrit_peak_ODE}, as is clearly
demonstrated in Figure~\ref{fig:supercrit_peak_rescaled}.

\end{enumerate}

We recall that calculation of the profile~$\BS$ requires knowing the numerical values
for~$\nu$ and~$\kappa$, see \myappendix\ref{sec:BS_numerics}.
The value of~$\kappa$ was obtained previously from the limit $\lim_{t \to T_c} (-4L^3L_t)^{1/4}$.
We approximate the value of the coefficient~$\nu$ from
\begin{equation}    \label{eq:nu_num}
    \nu_{\text{numeric}} = \lim_{t\to\TCrit}
            L^4(t)
            \frac{d \tau}{dt},\qquad 
            \tau = \arg \psi(t,r=0) .
\end{equation}
Indeed, Figure~\ref{fig:peak_L4phit} shows that in both cases~$
            L^4(t)
            \frac{d\tau}{dt}
$ 
quickly converge to 
\begin{equation} \label{eq:nu_numeric} 
    \nu_{\text{numeric}}(d=1,\sigma=6) = 0.36187, \qquad 
    \nu_{\text{numeric}}(d=2,\sigma=3) = 0.22826.
\end{equation}
 
As a further verification, using the above values for~$\nu$, we seek a value
of~$\kappa$ such that the solution of~\eqref{eq:supercrit_peak_ODE} will
satisfy~$|\BS(0)|=1$, and obtain
 \[
    \kappa(d=1,\sigma=6,\nu=0.36187) = 1.007,
    \quad 
    \kappa(d=2,\sigma=3,\nu=0.22826) = 0.894.
\]
These values are within 1\%--3\% from the values of~$1.037$ 
and~$0.909$ we obtained directly from the BNLS simulations, 
see~(\ref{eq:kappa_numeric}). 

Finally, we verified that the value of~$\kappa$ in the blowup
rate~\eqref{eq:rate_14} is universal.
We solve the BNLS in the case $d=1,\sigma=6$ with the initial
condition~$\psi_0(x)=2e^{-x^4}$.
In this case, the calculated value of~$\kappa(d=1,\sigma=6)$ is~$
    \kappa=\displaystyle\lim_{t\to\TCrit}\sqrt[4]{-4L_tL^3}\approx1.037
$, which is equal, to first~$3$ significant digits, to the previously obtained
value, see~\eqref{eq:kappa_d1s6}, for the initial condition~$\psi_0(x)=1.6e^{-x^2}$.
Similarly, in the case $d=2,\sigma=3$, we solve the equation with the initial
condition~$\psi_0(x)=3e^{-x^4}$. 
The calculated value of~$\kappa(d=2,\sigma=3)$ is~$
    \kappa=\displaystyle\lim_{t\to\TCrit}\sqrt[4]{-4L_tL^3}\approx0.913
$, which is equal, to first~$2$ significant digits, to the previously obtained
value, see~\eqref{eq:kappa_d2s3}, for the initial condition~$\psi_0(x)=3e^{-x^2}$.

\subsection*{Acknowledgments} 
We thank Nir Gavish for useful discussions, and Elad Mandelbaum for preliminary
numerical calculations.
This research was partially supported by grant \#123/2008 from the Israel
Science Foundation (ISF).

\appendix

\section{Numerical calculation of the $\BS$ profile}
\label{sec:BS_numerics}

\begin{subequations} \label{eqs:L} 
    In order to solve equation~\eqref{eq:supercrit_peak_ODE}, we first
    define its linear part, which is the fourth-order linear differential
    operator~$L\left[S\right]$ 
    \begin{eqnarray}
        L\left[S(\rho)\right] & = & 
        -\nu S(\rho) 
        + i\frac{\kappa^4}{4} \left(
            \frac2\sigma S + \rho S^{\prime}
        \right) 
        -\Delta_{\rho}^{2}S,\label{eq:L_op}
    \end{eqnarray}
    under the BCs, see equation~\eqref{eq:explicit_ODE_BCs_OH},
    \begin{equation}
    \begin{gathered}
        \label{eq:BCs-geom}
        S^{\prime}(0) = S^{\prime\prime\prime}(0)=S(\infty)=0,
        \qquad
        \lim_{\rho\to\infty} 
            \rho^{\gamma} \left( 
            \rho S^{\prime}+\left(
                \frac{2}{\sigma}+i\frac{4\nu}{\kappa^4} 
            \right)S
            \right)
            =0 \\
            \gamma_0<\gamma<\gamma_1,\qquad 
            \gamma_0 =\frac23\left( d-2-\frac2\sigma \right),\quad 
            \gamma_1 = 4+\frac2\sigma.
    \end{gathered}
    \end{equation}
\end{subequations}
The nonlinear ODE~\eqref{eq:supercrit_peak_ODE} is therefore rewritten as 
\begin{equation}
    L\left[S(\rho)\right]+\left|S\right|^{2\sigma}S = 0. \label{eq:profile}
\end{equation}

For given numerical values of~$\nu$ and~$\kappa$, we wish to calculate the
ground state of the nonlinear boundary-value problem~\eqref{eq:profile}.
In order to do so, we modify the SLSR method for the calculation of the
ground-state of the
NLS~\cite{petviashvili:1976,Pelinovsky:2004,Ablowitz-Musslimani-SLSR:2005} and
BNLS~\cite{Baruch_Fibich_Mandelbaum:2009a} as follows.
We consider the fixed-point iterative scheme
\begin{equation}    \label{eq:SLSR_unnormalized}
    S^{\left(k+1\right)}\left(\rho\right) 
        =-L^{-1}
            \left[\left|S^{\left(k\right)}\right|^{2\sigma}S^{\left(k\right)}\right],
            \qquad k=0,1,\dots
\end{equation}
for the solution of~\eqref{eq:profile}.
In the standard application of the SLSR method,~$L$ is a differential operator
of constant coefficients, and its inversion is easily performed using the
Fourier transform.
In our case,~$L$ is a variable-coefficient operator, and the Fourier Transform
cannot be used.
Therefore, we discretize the operator~$L$ using finite differences, see
\myappendix~\ref{ssec:L_r}, and invert it using the LU decomposition.

We observe numerically that generically, the iterations~\eqref{eq:SLSR_unnormalized} converge to zero for a small initial guess and diverge to infinity for a large initial guess.
To avoid this divergence, we rescale the approximate solutions at each
iteration, so that they satisfy the integral relation: 
\[
    \int|S|^{2}\rho^{d-1}d\rho =
        \langle S,S\rangle  
    = -\text{Re} \langle 
        S,L^{-1}|S|^{2\sigma}S
    \rangle, 
\]
which follows from multiplication of~\eqref{eq:SLSR_unnormalized} by~$S$. Here, $<\cdot>$~denotes the standard inner product $\langle f,g\rangle  = \int f^{*}g\rho^{d-1}d\rho$.
Following a similar argumentation as in~\cite{Baruch_Fibich_Mandelbaum:2009a},
we obtain that the iterations are
\begin{equation}    \label{eq:SLSR}
    S^{(k+1)} = -\left(
        \frac{
            -\langle S^{(k)},S^{(k)}\rangle 
        }{
            \text{Re}\left\langle
                S^{(k)}, L^{-1}|S^{(k)}|^{2\sigma}S^{(k)}
            \right\rangle
        }
    \right)^{1+\frac{1}{2\sigma}}
    L^{-1} \left[
        |S^{(k)}|^{2\sigma}S^{(k)}
    \right]. 
\end{equation}

In our simulation, this method converged for every value of~$\nu$
and~$\kappa$ that we tried.
The numerical values of~$\nu$ was obtained from the on-axis phase of the BNLS
simulation solutions, as explained in Section~\ref{ssec:simulations}.
In order to obtain a prediction of~$\kappa$, we recall that the specific
choice~\eqref{eq:psi_rescaled_peak} of the blowup rate~$L(t)$ implies
that~$|\BS(0)|=1$.
We therefore use the SLSR solver to search for the value of~$\kappa$ for
which~$|\BS(0)=1|$.

\subsection{\label{ssec:L_r} Discretization of~$L$}

Using half-integer grid \[
    \rho_{n}=\left(n+\frac{1}{2}\right)h,
    \qquad n=0,\dots N-1,
    \qquad h=\frac{R_{\max}}{N},
\]
the ${\cal O}\left(h^{2}\right)$ centered-difference discretizations
of the radial biharmonic operator $D_{\rho}^{2}$ and of the first-derivative,
the approximation at the interior nodes is 
\begin{eqnarray*}
    \left(-\nu+\frac{i\kappa^4}{2\sigma}\right)S_{n}
    +\frac{i\kappa^4}{4}\rho_{n}\frac{S_{n+1}-S_{n-1}}{2h}
    -D_{\rho}^{2}S_{n}
    +\left|S_{n}\right|^{2\sigma}S_{n} 
    & = & {\cal O}\left(h^{2}\right).
\end{eqnarray*}

The stencil is five-nodes wide, so two ghost-nodes are needed at each
boundary.
In order to enfold the ghost-nodes at~$\rho=0$, we relate them to the interior
nodes, using the symmetry of the solution~$S(\rho)=S(-\rho)$, so that \[
    \left[\begin{array}{c} S_{-2} \\ S_{-1} \end{array}\right]
    =
    \left[\begin{array}{cc} 0 & 1 \\ 1 & 0 \end{array}\right]
        \left[\begin{array}{c} S_{0} \\ S_{1} \end{array}\right].
\]
This relation is substituted in the discretization of the equation at~$\rho_0$
and~$\rho_1$.

At the other boundary $\rho=R_{\max}$ we use the approximate form of the
solution obtained from the WKB approximation, i.e., we require 
\[
S_n = 
    c_1 \BS_{,1}(\rho_n)
    +c_4 \BS_{,4}(\rho_n),
    \qquad n= N-2,N-1,N,N+1,\dots .
\]
In matrix form, this becomes  
\begin{equation*}
    \left[\begin{array}{cc}
        \BS_{,1}(\rho_{N-1}) & \BS_{,4}(\rho_{N-1})\\
        \BS_{,1}(\rho_{N}) & \BS_{,4}(\rho_{N})\\
        \BS_{,1}(\rho_{N+1}) & \BS_{,4}(\rho_{N+1})\\
        \BS_{,1}(\rho_{N+2}) &
        \BS_{,4}(\rho_{N+2})
    \end{array}\right]
    \left[\begin{array}{c} c_1\\ c_4 \end{array}\right]
    = \left[\begin{array}{c}
        S_{N-1}\\
        S_{N}\\
        S_{N+1}\\
        S_{N+2}
    \end{array}\right]
\end{equation*}
which is then solved to obtain \[
    \left[\begin{array}{c}
        S_{N+1}\\
        S_{N+2}
    \end{array}\right]
    =
    \left[\begin{array}{cc}
        \BS_{,1}(\rho_{N+1}) & \BS_{,4}(\rho_{N+1})\\
        \BS_{,1}(\rho_{N+2}) & \BS_{,4}(\rho_{N+2})
    \end{array}\right]
    \left[\begin{array}{cc}
        \BS_{,1}(\rho_{N-1}) & \BS_{,4}(\rho_{N-1})\\
        \BS_{,1}(\rho_{N}) & \BS_{,4}(\rho_{N})
    \end{array}\right]^{-1}
    \left[\begin{array}{c}
        S_{N-1}\\
        S_{N}
    \end{array}\right].
\]

Some care should be taken when choosing the parameters~$R_{\max}$ and~$N$.
On the one hand, we use the closed-form approximations for~$\BS_{,1}$
and~$\BS_{,4}$ that become more accurate for~$R_{\max}\gg1$.
On the other hand, since~$\BS_{,4}$ has a super-exponentially decreasing
term~$e^{-\rho^{4/3}}$, choosing too-large a value of~$R_{\max}$ leads to
numerical instabilities.
Finally, in order to resolve the rapid-oscillations~$e^{i\rho^{4/3}}$
of~$\BS_{,4}$, the grid-size~$N$ must be chosen such that~$
    \rho^{1/3}h_{\rho} 
    = \mathcal{O}\left( R_{\max}^{4/3}/N \right) \ll 1
$, hence that~$\left.N \geq \mathcal{O}\left( R_{\max}^{4/3} \right)\right.$.
The grid-size~$N$, however, cannot be arbitrarily large, since the condition
number of~$L$ is~$\mathcal{O}\left( N^4 \right)$.

In the simulations presented in this study, we used an extension of  above approach to a fourth-order approximation,
and set~$R_{\max}=160$ and~$N=32000$.

~


\end{document}